\documentclass[11pt]{article}

\usepackage{amstext,amssymb,amsmath,amsbsy}
\usepackage{hyperref}
\usepackage{amscd}
\usepackage{amsfonts}
\usepackage{indentfirst}
\usepackage{verbatim}
\usepackage{amsmath}
\usepackage{amsthm}
\usepackage{enumerate}
\usepackage{graphicx}
\usepackage{color}
\usepackage[OT1]{fontenc}
\usepackage[latin1]{inputenc}
\usepackage[english]{babel}
\usepackage{amssymb}
\usepackage{subfig}
\usepackage{algorithm}
\usepackage{algpseudocode}
\usepackage[shortlabels]{enumitem}

\newcommand{\R}{\mathbb{R}}
\newcommand{\ds}{\displaystyle}

\newcommand{\x}{{\bf x}}

\newtheorem{Theorem}{Theorem}[section]
\newtheorem{Lemma}{Lemma}[section]

\newtheorem{Proposition}{Proposition}[section]
\newtheorem{Corollary}{Corollary}[section]
\newtheorem{Remark}{Remark}[section]

\newtheorem*{Assumption*}{Assumption}

\newtheorem{Problem}{Problem}[section]
\newtheorem*{Problem*}{Problem}
\numberwithin{equation}{section}

\begin{document}

\title{Global reconstruction of initial conditions of  nonlinear parabolic equations via the Carleman-contraction method}

\author{Thuy T. Le\thanks{%
Department of Mathematics and Statistics, University of North Carolina at
Charlotte, Charlotte, NC 28223, USA, \text{tle55@uncc.edu}.}}

\date{}
\maketitle
\begin{abstract}
	We propose a global convergent numerical method to reconstruct the initial condition of a nonlinear parabolic equation from the measurement of both Dirichlet and Neumann data on the boundary of a bounded domain. 
The first step in our method is to derive, from the nonlinear governing parabolic equation, a nonlinear systems of elliptic partial differential equations (PDEs) whose solution yields directly the solution of the inverse source problem. We then establish a contraction mapping-like iterative scheme to solve this system.  The convergence of this iterative scheme is rigorously proved by employing a Carleman estimate and the argument in the proof of the traditional contraction mapping principle.
This convergence is fast in both theoretical and numerical senses.  
Moreover, our method, unlike the methods based on optimization, does not require a good initial guess of the true solution.
Numerical examples are presented to verify these results.
\end{abstract}



\section{Introduction}
The problem of recovering the initial condition for a nonlinear parabolic equation from the measurement of the Dirichlet and Neumann data in a bounded domain arises in many real-world applications, for example, detecting the pollution on the surface of the rivers or lakes \cite{BadiaDuong:jiip2002}, reconstructing of the spatially distributed temperature inside a solid from the heat and the heat flux on the boundary in the time domain \cite{Klibanov:ip2006}, effective monitoring heat and conduction processes in steel industries, glass and polymer-forming and nuclear power station \cite{LiYamamotoZou:cpaa2009}. Due to its realistic applications, this inverse source problem has been studied intensively. It is formulated as follows.

Let $d \geq 1$ be the spatial dimension,
let $\Omega$ be an open and bounded domain  in $\R^d$ with smooth boundary $\partial \Omega$, and let $T$ be a positive number. 
Let $F: \R^d \times \R \times \R \times \R^d \to \R$ and $c: \R^d \to [c_0, \infty)$, for some $c_0 > 0$,  be functions in the class $C^1.$
 Denote by $Q_T$ the set $\Omega \times (0,T)$.
 
Consider the function $u(\x,t)$ is governed by the following problem:
\begin{equation}
	\left\{
		\begin{array}{ll}
			c(\x)u_t(\x, t) = \Delta u(\x, t) + F\left(\x, t, u(\x,t),\nabla u(\x,t)\right) &(\x,t) \in \R^d \times (0, T) \\
			u(\x,0) = p(\x) & \x \in \R^d,
	\end{array}
	\right.
	\label{main eqn}
\end{equation}
where $p \in C_{\rm c}^1(\R^d)$ is a source function. Assume that $p$ is compactly supported in $\Omega.$
Since the main aim of this paper concerns with inverse problem, not forward problem, the existence, uniqueness and regularity results for the solution to \eqref{main eqn} are considered as  assumptions. 
For the completeness, we provide here a set of conditions that guarantees that these assumptions hold true.
Assume that  $p(\x)$ is in $H^{2+\beta}(\R^d)$ for some $\beta \in [0, 1 + 4/d]$. Assume further  for all 
$\x \in \R^d$, $t \in [0, T]$, $s \in \R$, and $r \in \R^d$,
\[
	|F(\x, t, s, r)| \leq C\max\big\{
		(1 + |r|)^2, 1 + |s|
	\big\}
\]
for some positive constant $C$.
Then, due to Theorem 6.1 in \cite[Chapter 5, \S 6]{LadyZhenskaya:ams1968} and Theorem 2.1 in \cite[Chapter 5, \S 2]{LadyZhenskaya:ams1968}, 
problem \eqref{main eqn} has a unique solution with $|u(\x,t)|\leq M_1$, $|\nabla u(\x,t)| \leq M_2$ for some positive constants $M_1,M_2$, and $u(\x,t) \in H^{2+\beta, 1+\beta/2}(\R^d \times [0,T])$.

We are interested in the problem of determining the source function $p(\x), \x \in \Omega,$ from the lateral Cauchy data. 
This problem is formulated as follows.
\begin{Problem}[Inverse Source Problem]	
	Assume that $\|u\|_{C^1(\overline{\Omega_T})} \leq M$ for some known large number $M$.
	Given the lateral Cauchy data 
	\begin{equation}
		g(\x, t) = u(\x, t) 
		\quad
		\mbox{and }
		\quad
		q(\x, t) = \partial_\nu u(\x, t)
		\label{data}
	\end{equation}
	for $\x \in \partial \Omega$, $t \in [0, T]$,
	determine the function $u(\x, 0) = p(\x), \x \in \Omega.$
	\label{ISP}
\end{Problem}

The uniqueness of Problem \ref{ISP} is still an open question. 
We temporarily consider the uniqueness as an assumption. 
This result will be written in a separate paper.
In the case when \eqref{main eqn} is linear, the uniqueness of Problem \ref{ISP} is well-known, see \cite{Lavrentiev:AMS1986}. The logarithmic stability results in the linear case were rigorously proved in \cite{Klibanov:ip2006, LiYamamotoZou:cpaa2009}.

A natural approach to solve the Problem \ref{ISP} are based on least squares optimization. However, since our problem is highly nonlinear, the cost functional might have multiple minima and ravines. Thus, optimization-based methods might not deliver good solutions; especially, when a good initial guess is unavailable. 
An effective way to overcome the lack of a good initial guess is the convexification method, which first introduced in \cite{KlibanovIoussoupova:siam1995}. The main idea of the convexification method  is to convexify the cost functional by employing a suitable Carleman weight function. We refer the reader to  \cite{Klibanov:siam1997,KlibanovNik:ra2017, KhoaKlibanov:ip2020, KhoaKlibanov:siam2020, Klibanov:jiip2017, KhoaKlibanov:ipse2021, Klibanov:cma2019, Klibanov:siam2019, Klibanov:ip2018, KlibanovLiJingzhi:ip2019, KlibanovLiJingzhi:ip2020, KlibanovLiJingzhi:siam2019, SmirnovKlibanov:siam2019, LeNguyen:arxiv2021, KlibanovLeNguyen:ipi2021} for the development of the convexification method. 
Although effective, the convexification method is time consuming. 
Therefore, another method should be investigated.

In this paper, we introduce an iterative method to solve Problem \ref{ISP}. The first step in our numerical method is to derive a system of quasi-linear elliptic equations whose solution yields directly the solution of Problem \ref{ISP} by truncating a Fourier series with respect to the special basis in \cite{Klibanov:jiip2017}. 
In the second step, we propose a fixed point-like iterative process to solve the system mentioned above. 
The iterative process can start from an arbitrarily function.
The idea to design the iterative process is similar to \cite{LeNguyen:jiip2020, LeKlibanovNguyen:ip2022, NguyenKlibanov:ip2022, LocNguyen:arXiv2022}. 
Especially, in \cite{LeNguyen:jiip2020}, we study a similar problem to recover the initial condition of a nonlinear parabolic equation from the lateral Cauchy data. The main difference between \cite{LeNguyen:jiip2020} and this paper is that in this paper, we study a more complicated case at which the nonlinear term $F$ depends on both $u(\x,t)$ and $\nabla u(\x,t)$. 
The convergence of our scheme is rigorously proved based on a Carleman estimate and the analogue contraction mapping principle. Our iterative scheme converges quickly to the true solution at the rate $\theta^n$ for some constant $\theta \in (0,1)$ and $n$ is the number of iterations. In particular, the convergence of our numerical method is rigorously proved in $H^2$. This result is a significant improvement in comparison to the main theorem in \cite{NNT:arxiv2022, LocNguyen:arXiv2022} which shows a similar convergence  in $H^1$ only.
These results is verified rigorously by our numerical tests.

The paper is organized as follows. In Section 2, we recall some preliminaries which we employ directly in our numerical method, including a Carleman estimate and a special orthonormal basis. In Section 3, we introduce two steps to solve Problem \ref{ISP}. The first step is to derive a system of nonlinear PDEs whose solutions yields directly the solutions to Problem \ref{ISP}. The second step is to establish an iterative scheme to solve the above system of nonlinear PDEs. In Section 4, we prove the convergence of our iterative scheme to the true solution. In Section 5, we discuss the implementation of our method and present some numerical results. Section 6 is some concluding remarks.

\section{Preliminaries}
\label{sec Preliminaries}
In this section, we recall a Carleman estimate established in \cite{LeNguyen:jiip2020}.
This Carleman estimate plays an important role this paper.
On the other hand, we  recall a special orthonormal basis $\{\Phi_n\}_{n\geq 1}$ of $L^2(0,T)$, which will be used in the numerical implementation section, Section \ref{sec numerical study}.  This special basis was first introduced in \cite{Klibanov:jiip2017}.

\subsection{A Carleman estimate}
\label{sec Carleman estimate}

The Carleman estimate is a powerful tool in the field of PDEs which were first employed to prove the unique continuation principle, see \cite{Carleman:1933, Protter:1960}, the uniqueness of a long list of inverse problems, see \cite{BukhgeimKlibanov:smd1981} and in cloaking \cite{MinhLoc:tams2015}. Here we recall a  Carleman estimate which was established in \cite{LeNguyen:jiip2020}. The analysis of this paper is based on this Carleman estimate.

\begin{Lemma}[Carleman estimate, see  \cite{LeNguyen:jiip2020}]
	Let $\x_0$ be a point in $\R^d \setminus \overline \Omega$ such that $r(\x) = |\x - \x_0| > 1$ for all $\x \in \Omega$.
	Let $b > \max_{\x \in \overline \Omega} r(\x)$ be a fixed constant.
	There exist positive constants $\beta_0$  depending only on $b$, $\x_0$, $\Omega$ and $d$ such that
	for all function $v \in C^2(\overline \Omega)$ satisfying 
	 \[
	 	v(\x) = \partial_{\nu} v(\x) = 0  \quad \mbox{for all } \x \in \partial \Omega,
	 \]
	the following estimate holds true
\begin{multline}
	\int_{\Omega} e^{2\lambda b^{-\beta} r^{\beta}(\x)}|\Delta v(\x)|^2 d\x
	\geq 
	\frac{C}{\lambda \beta^{7/4} b^{-\beta}} \int_{\Omega}e^{2\lambda b^{-\beta} r^\beta(\x)} r^{2\beta}(\x)|D^2v(\x)|^2 d\x
	\\
	+ 	C \lambda^3 \beta^4 b^{-3 \beta} \int_{\Omega} r^{2\beta}(\x) e^{2\lambda b^{-\beta} r^{\beta}}|v(\x)|^2 d\x
	\\
	+ C \lambda \beta^{1/2} b^{-\beta}\int_{\Omega} e^{2\lambda b^{-\beta} r^{\beta}(\x)} |\nabla v(\x)|^2 d\x
	\label{Car est}
\end{multline}
	for $\beta \geq \beta_0$ and $\lambda \geq \lambda_0$.
		Here, $D^2 v = (v_{x_i x_j} )_{i, j = 1}^d$ is the Hessian matrix of $v$, $\lambda_0 = \lambda_0(b, \Omega, d, \x_0) > 1$ is a positive number with $\lambda_0 b^{-\beta} \gg 1$ and $C = C(b, \Omega, d, \x_0) > 1$ is a constant.  These numbers depend only on listed parameters.
	 \label{carleman estimate 1}
\end{Lemma}	

\begin{Corollary}
	Recall $\beta_0$ and $\lambda_0$ as in Lemma  \ref{carleman estimate 1}.
	Fix $\beta = \beta_0$ and let the constant $C$ depend on $\x_0,$ $\Omega,$ $d$ and $\beta$. 
	There exists a constant $\lambda_0$ depending only on $\x_0,$ $\Omega,$ $d$ and $\beta$ such that
  for all function $v \in H^2(\Omega)$ with 
	\[
		v(\x) = \partial_{\nu} v(\x) = 0 
		\quad \mbox{ on } \partial \Omega,
	\] 
	we have
	\begin{multline}
	\int_{\Omega} e^{2\lambda b^{-\beta} r^{\beta}(\x)}|\Delta v(\x)|^2 d\x
	\geq 
	C\lambda^{-1}\int_{\Omega}e^{2\lambda b^{-\beta} r^\beta(\x)} |D^2 v(\x)|^2 d\x
	\\
	+ 	C \lambda^3 \int_{\Omega}  e^{2\lambda b^{-\beta} r^{\beta}}|v(\x)|^2 d\x
	+ C \lambda \int_{\Omega} e^{2\lambda b^{-\beta} r^{\beta}(\x)} |\nabla v(\x)|^2 d\x	
	\label{33}
\end{multline}
	for all $\lambda \geq \lambda_0$.
	\label{carleman estimate}
\end{Corollary}

We refer the reader to \cite[Theorem 3.1]{LeNguyen:jiip2020} for the proof of Lemma \ref{carleman estimate 1}.
We also refer the readers to \cite{LiNguyen:IPSE2019, NguyenLiKlibanov:IPI2019} for some versions of Carleman estimates for parabolic operators, and to \cite{BeilinaKlibanovBook, BukhgeimKlibanov:smd1981, KlibanovLi:book2021} for several other versions of Carleman estimates for a variety kinds of differential operators and their applications in inverse problems.

\subsection{A special orthonormal basis}
\label{sec basis}

In this section, we recall a special basis of $L^2(0,T)$. This special basis will be employed in our numerical study.
Let $\phi_n(t) = (t-T/2)^{n-1}e^{t-T/2}$ for $n \in \mathbb{N}$. The set $\{\phi_n\}_{n=1}^{\infty}$ is complete in $L^2(0,T)$. Applying the GramSchmidt orthonormalization process to this complete set gives a basis of $L^2(0,T)$, named as $\{\Psi_n\}_{n=1}^{\infty}$. 
The following proposition holds true.
\begin{Proposition}[see \cite{Klibanov:jiip2017}]
The basis $\{\Psi_n\}_{n=1}^{\infty}$ satisfies the following properties:
\begin{enumerate}[(i)]
	\item $\Psi_n'$ is not identically zero for all $n \geq 1$
	\item For all $m, n \geq 1$
		\begin{equation}
		s_{mn} = \int_0^T \Psi_n'(t)\Psi_m(t) = 
			\begin{cases}
				1 \mbox{ if } m=n, \\
				0 \mbox{ if } n<m.
			\end{cases}
		\end{equation}
\end{enumerate}
As a results, for all integer $N > 1$, the matrix $S = (s_{mn})_{m,n=1}^N$ is invertible.
\end{Proposition}

This basis was first introduced to solve the electrical impedance tomography problem with partial data in \cite{Klibanov:jiip2017}. Afterward, it is widely used in our research group to solve a variety kinds of inverse problems, including ill-posed inverse source problems for elliptic equations \cite{NguyenLiKlibanov:IPI2019}, parabolic equations \cite{LeNguyen:jiip2020} \cite{LiNguyen:IPSE2019} and hyperbolic equations \cite{LeNguyenNguyenPowell:jsc2021}, nonlinear coefficient inverse problems for elliptic equations \cite{KhoaKlibanov:ip2020}, and parabolic equations \cite{KlibanovNguyen:ip2019,KhoaKlibanov:siam2020,
Nguyen:cma2020,KhoaKlibanov:ipse2021,LeNguyen:arxiv2021}, transport equations \cite{KlibanovLeNguyen:siam2020} and full transfer equations \cite{SmirnovKlibanov:siam2019}.

\section{A numerical method to solve Problem \ref{ISP} }
\label{sec numerical method}	
In this section, we present our numerical method to solve Problem \ref{ISP}.
Our method consists of two (2) main steps. Firstly, we derive a nonlinear system of elliptic equations  by cutting of a Fourier series with respect to an orthonormal basis of $L^2(0,T)$.  Solution of this system directly leads to that of Problem \ref{ISP}. Secondly, we propose a fixed point-like iterative scheme to solve the nonlinear system mentioned in step 1.

\subsection{A system of nonlinear elliptic equations}

Let $\{\Psi_n\}_{n=1}^{\infty}$ be an orthonormal basis of $L^2(0,T)$. For all $(\x,t) \in Q_T$, we can approximate $u(\x,t)$ as follows.
\begin{equation}
	u(\x,t) = \sum_{n=1}^{\infty} u_n(\x) \Psi_n(t) 	
	\label{3.1}
\end{equation}
where 
\begin{equation}
	u_n(\x) = \int_0^T u(\x,t) \Psi_n(t) dt \quad \mbox{ for all } n\geq 1.
	\label{approx un}
\end{equation}

We now derive our approximation model. For some cut-off number $N$, chosen  later in Section \ref{sec numerical study}, we approximate the function $u(\x, t)$ by truncating the series in \eqref{3.1} as
\begin{equation}
	u(\x, t) \approx u^N(\x,t) := \sum_{n=1}^N u_n(\x) \Psi_n(t) \quad \mbox{ for all } (\x,t) \in Q_T
		\label{approx u}
\end{equation}
We also approximate
\begin{equation}
	u_t(\x, t) \approx u_t^N(\x,t) := \sum_{n=1}^N u_n(\x) \Psi'_n(t) \quad \mbox{ for all } (\x,t) \in Q_T
	\label{approx ut}
\end{equation}

\begin{Remark}
	The approximation in \eqref{approx u} is numerically verified by Figure \ref{fig choose N}.
\end{Remark}

Plugging (\ref{approx u}) and (\ref{approx ut}) into the governing equation (\ref{main eqn}), we obtain \begin{equation}
	c(\x) \sum_{n=1}^N u_n(\x) \Psi'_n(t) = \sum_{n=1}^N \Delta u_n(\x) \Psi_n(t)
	+ F\left(\x,t,\sum_{n=1}^N u_n(\x)\Psi_n(t), \sum_{n=1}^N \nabla u_n(\x) \Psi_n(t) \right)
	\label{approx model}
\end{equation}
for all $(\x,t) \in Q_T$.
\begin{Remark}
The cut-off number $N$ is chosen numerically such that $u^N$ well approximates the function u, see Section \ref{sec implementation} for more details. Due to the nonlinearity and the ill-posedness of this inverse source problem, studying the convergence of (\ref{approx model}) as $N \rightarrow \infty$ is extremely challenging and out of scope of this paper. We only solve Problem \ref{ISP} in the approximation context.
\end{Remark}

For each $m \in \{1, \dots,N \}$,  multiplying $\Psi_m(t)$ to both sides of (\ref{approx model}) and then integrating the obtained equation with respect to $t$ give
\begin{multline}
	c(\x)\sum_{n=1}^N u_n(\x) \int_0^T \Psi'_n(t) \Psi_m(t) dt 
	 = \sum_{n=1}^N \Delta u_n(\x) \int_0^T \Psi_n(t) \Psi_m(t) dt \\
	 + \int_0^T F\left(\x,t,\sum_{n=1}^N u_n(\x)\Psi_n(t), \sum_{n=1}^N \nabla u_n(\x) \Psi_n(t) \right) \Psi_m(t)dt.
	 \label{3.6}
\end{multline}

Since $\{\Psi_n\}_{n \geq 1}$ is an orthonormal basis,  system (\ref{3.6}) can be rewritten as
\begin{equation}
	c(\x) \sum_{n=1}^N s_{mn}u_n(\x) 
	= \Delta u_m(\x) + F_m\left(\x, U(\x),\nabla U(\x)\right)
	\label{3,7}
\end{equation}
for all $\x \in \Omega$ and
 $m = 1, \dots, N$ where $U = (u_1, u_2, \dots, u_N)^{\rm T}$,
\begin{equation*}
s_{mn} = \int_0^T \Psi'_n(t) \Psi_m(t)dt
\label{smn}
\end{equation*}
and 
\begin{equation*}
F_m\left(\x, U(\x),\nabla U(\x)\right)
= \ds\int_0^T F\left(\x,t,\sum_{n=1}^N u_n(\x)\Psi_n(t), \sum_{n=1}^N \nabla u_n(\x) \Psi_n(t) \right) \Psi_m(t)dt.
\end{equation*}
Denote by $S$ the matrix $(s_{mn})_{m,n = 1}^N$ and ${\bf F} = (F_1, F_2, \dots, F_N)^{\rm T}$. We can rewrite \eqref{3,7} as
\begin{equation}
	\Delta U(\x) - c(\x) SU(\x) + {\bf F}(\x, U(\x), \nabla U(\x)) = 0
	\quad \mbox{for all } \x \in \Omega.
	\label{3.7}
\end{equation}

We next compute the Cauchy boundary conditions for $U_N$  Due to \eqref{approx un},
\begin{equation}
\begin{cases}
	U(\x) = G(\x) = \Big(\ds\int_0^T g(\x,t) \Psi_m(t)dt\Big)_{m = 1}^N, \\
	\partial_{\nu} U(\x) = Q(\x) = \Big(\ds\int_0^T q(\x,t) \Psi_m(t)dt\Big)_{m = 1}^N
\end{cases}
\quad m = 1,2,\dots,N
\label{3.8}
\end{equation}
for all $\x \in \partial\Omega$, $m= 1,\dots,N$, where  $g(\x,t)$ and $q(\x,t)$ are the given data in Problem \ref{ISP}.

Combining (\ref{3.7}) and (\ref{3.8}), we obtain the system of elliptic equations for $u_m(\x)$ for each $m = 1, \dots,N$
\begin{equation}
	\left\{
		\begin{array}{ll}
			\Delta U(\x) - c(\x) SU(\x) + {\bf F}(\x, U(\x), \nabla U(\x)) = 0& \x \in \Omega,
			\\
			U(\x) = G(\x) &\x \in \partial\Omega,\\
		\partial_{\nu} U(\x) = Q(\x) &\x \in \partial\Omega.
	\end{array}
	\right.
	\label{main sys}
\end{equation}

\begin{Remark}
Due to the truncation in \eqref{approx u}, Problem \eqref{main sys} is not exact. It is an approximation model.
Proving the convergence of this approximation model as $N \to \infty$ is extremely challenging.
Establishing this result is out of the scope of this paper.
The accuracy of \eqref{approx u} can be verified numerically, see Figure \ref{fig choose N}.
\end{Remark}

Problem \ref{ISP} is reduced to the problem of finding  $U = (u_1, u_2, \dots, u_N)^{\rm T}$ satisfying  system \eqref{main sys}. 
In fact, if this vector is known, the function $u(\x,t)$ for all $(\x,t) \in Q_T$ can be approximated via (\ref{approx u}). 
Then, the solution to Problem \ref{ISP} is given by $p(\x) = u(\x, 0)$ for all $\x \in \Omega.$
In the next section, we introduce an iterative procedure to solve  system \eqref{main sys}.

\subsection{An iterative procedure to solve  system (\ref{main sys})}
\label{iterative scheme}
We introduce an iterative scheme to solve  system \eqref{main sys}.
The convergence of this scheme to the true solution of \eqref{main sys} will be discussed later in Section \ref{sec conv}.

Let $U^{(0)}$ be an arbitrary vector-valued function.
Assume by induction that we know $U^{(k -1)}$ for $k\geq 1$. We then find $U^{(k)}$ by solving the following system
\begin{equation}
	\left\{
		\begin{array}{ll}
			\ds \Delta U^{(k)}(\x) - c(\x) SU^{(k)}(\x) = - {\bf F}(\x, U^{(k - 1)}(\x), \nabla U^{(k-1)}(\x)) &\x \in \Omega,
			\\
			U(\x) = G(\x) &\x \in \partial\Omega,\\
		\partial_{\nu} U(\x) = Q(\x) &\x \in \partial\Omega.
	\end{array}
	\right.
	\label{iter sys}
\end{equation}

 Problem \eqref{iter sys} might not have a solution because it is over-determined. We only compute the ``Carleman best-fit" solution $U^{(k)}$ to \eqref{iter sys} by combining the quasi-reversibility method and a Carleman weight function as follows.
 The convergence of the sequence of ``Carleman best-fit" solutions is one of the important strengths  of this paper.
	 Define the set of admissible solutions
\begin{equation}
H = \left\lbrace V \in H^2(\Omega)^N: 
\right.
V|_{\partial \Omega} = G \mbox{ and }
\left. \partial_{\nu}V|_{\partial \Omega} = Q \right\rbrace.
\label{H}
\end{equation}
Throughout this paper, we assume that $H$ is nonempty.
Recall $\beta_0$ and $\lambda_0$ as in Lemma \ref{carleman estimate 1}.
Fix $\beta > \beta_0$. 
For each $\lambda > \lambda_0$,
we define the following Carleman weighted least squares functional
 \begin{equation}
J^{(k)}\left(V\right) = \ds\int_{\Omega} e^{2\lambda b^{-\beta} r^{\beta}(\x)} \bigg| \Delta V - c(\x)SV +
{\bf F}(\x, U^{(k - 1)}, \nabla U^{(k-1)}) \bigg|^2 d\x 
\label{Jk}
\end{equation}
for $V \in H.$ Then, we set 
\begin{equation}
U^{(k)} = {\rm argmin}_{V \in H} J^{(k)}(V).
\label{unk}
\end{equation}
See Theorem \ref{unique minimizer} for the existence and uniqueness of the minimizer of $J^{(k)}$.

Recall that the function $U^{(k)}$ is called the ``Carleman best fit" solution to \eqref{iter sys} obtained by the Carleman quasi-reversibility method.
 The original quasi-reversibility method was  introduced by Latt\`{e}s and Lions in 1969, see \cite{LattesLions:book1969}. We refer readers to \cite{LiNguyen:IPSE2019, NguyenLiKlibanov:IPI2019, LeNguyenNguyenPowell:jsc2021} for using the quasi-reversibility method to solve a linear system of PDEs and \cite{Klibanov:anm2015} for a survey of the quasi-reversibility method. In the original quasi-reversibility method, the best fit solutions to \eqref{iter sys} can be found by minimizing the least-squares functional. Here, we improve this method by imposing a Carleman weight function on the original least-squares functional.
 By ``improve", we mean that the presence of the Carleman weight function plays a crucial role in the convergence analysis in the next section.
 More precisely, we will show that the sequence $\{U^{(k)}\}_{k \geq 0}$ converges to the true solution to \eqref{main sys} as $k$ goes to $\infty$. The choice of the initial term does not matter.

We summary the procedure to solve Problem \ref{ISP} in the Algorithm.
\begin{algorithm}[!ht]
\caption{\label{alg}The procedure to solve Problem \ref{ISP}}
	\begin{algorithmic}[1]
	\State \label{step1} Choose an orthonormal basis $\left(\Psi_n\right)_{n\geq1}$ and a cut-off number N.
	\State \label{step2} Compute the Cauchy data $G$ and $Q$ on $\partial \Omega$ as in (\ref{3.8}).
	\State \label{step3} Choose an arbitrary vector valued function $U^{(0)} \in H^2(\Omega)^N.$
	\State \label{step4} By induction, we assume that we known $U^{(k-1)}$ $k\geq 1$.
	Solve the system (\ref{iter sys}) by Carleman quasi-reversibility method  for the vector valued function $U^{(k)}= \left( u_1^{(k)},\dots,u_m^{(k)} \right)^{\rm T}$.
	\State  \label{step6} \label{find source}Set the computed source function at step $k$ is 
			\[ p^{(k)}(\x) = u^{(k)}(\x,0) = \sum_{n = 1}^N u^{(k)}_n(\x)\Psi_n(0)\]
	\State \label{pcomp}Set the computed source function is 
	$ p_{\rm comp} = p^{(k)}$ for $k=\overline k$ large enough. 
\end{algorithmic}
\end{algorithm}

\section{The main theorems} \label{sec conv}



We establish two theorems in this section.
The first one proves that the sequence $\{U_N^{(k)}\}_{k \geq 0}$, constructed in Section \ref{iterative scheme} and \eqref{unk},  is well-defined. The second one guarantees the convergence of $\{U_N^{(k)}\}_{k \geq 0}$.

\begin{Theorem}
Fix $\beta = \beta_0$ and choose $\lambda > \lambda_0$, where $\beta_0$ and $\lambda_0$ are in Lemma \ref{carleman estimate 1} and Corollary \ref{carleman estimate}, such that \eqref{33} holds true.
For each $k \geq 1$, the functional $J^{(k)}$, defined in \eqref{Jk}, has a unique minimizer in $H$. 
\label{unique minimizer}
\end{Theorem}

\begin{proof}
The proof of Theorem \ref{unique minimizer} is based on the Riesz Representation Theorem.
Since the set of admissible solutions $H$ is nonempty, we can find a vector valued function  $\Phi = (\phi_1, \dots, \phi_N)^{\rm T}$  in $H$.
Define 
\begin{equation}
H_0 =  \big\{ U - \Phi \in H^2(\Omega)^N: U \in H \big\} \
\label{H0}
\end{equation}
It is obvious that for all $V = (v_1, \dots, v_N)^{\rm T} \in H_0$, $V|_{\partial \Omega} = 0$ and $\partial_{\nu}V|_{\partial \Omega} = 0.$
Using the trace theory, we can see that $H_0$ is a closed subspace of $H^2(\Omega)^N.$

Define the bounded bilinear form $a: H_0 \times H_0 \to \R$ and the bounded linear map $L: H_0 \to \R$ as 
\begin{equation*}
a(P,Q) =  \ds\int_{\Omega} e^{2\lambda b^{-\beta} r^{\beta}(\x)} \big( \Delta P - c(\x) S P \big) \cdot \big( \Delta Q - c(\x)SQ \big) d\x
\end{equation*}
and 
\begin{equation}
LQ = - \int_{\Omega} e^{2\lambda b^{-\beta} r^{\beta}(\x)} \big( \Delta \Phi(\x) - c(\x)S\Phi 
+ {\bf F}(\x, U^{(k - 1)}, \nabla U^{(k-1)})\big) \cdot \big(
 \Delta Q(\x) - c(\x)SQ\big) d\x
 \label{LQ}
\end{equation}
for all $P, Q \in H_0$.

We will show that $a(P,Q)$ is an inner product on $H_0$ and equivalent to the standard norm in $H^2(\Omega)^N$.
In fact,
for all $P \in H_0$,
\begin{equation}
	a(P,P) = \ds\int_{\Omega} e^{2\lambda b^{-\beta} r^{\beta}(\x)} | \Delta P - c(\x)SP |^2 d\x
	\leq C \Vert P \Vert_{H^2(\Omega)^N}^2.
	\label{4,4}
\end{equation}
On the other hand, for all $P \in H_0$, using the inequality $(x - y)^2 \geq \frac{1}{2}x^2 - y^2$, we have
\begin{multline}
	a(P,P) = \ds\int_{\Omega} e^{2\lambda b^{-\beta} r^{\beta}(\x)} | \Delta P - c(\x)SP |^2 d\x \\
	\geq \frac{1}{2}\ds\int_{\Omega} e^{2\lambda b^{-\beta} r^{\beta}(\x)} | \Delta P |^2 d\x - \ds\int_{\Omega} e^{2\lambda b^{-\beta} r^{\beta}(\x)} |c(\x)SP |^2 d\x.
\end{multline}
Applying the Carleman estimate \eqref{33} for vector $P \in H_0$, we have
\begin{multline}
	a(P,P) 
	\geq C\lambda^{-1} \int_{\Omega}e^{2\lambda b^{-\beta} r^\beta(\x)} |D^2P|^2 d\x
	\\
	+ 	C \lambda^3 \int_{\Omega}  e^{2\lambda b^{-\beta} r^{\beta}}|P|^2 d\x 
	+ C \lambda \int_{\Omega} e^{2\lambda b^{-\beta} r^{\beta}(\x)} |\nabla P|^2 d\x	\\
	-  \ds\int_{\Omega} e^{2\lambda b^{-\beta} r^{\beta}(\x)} |c(\x) SP |^2 d\x.
\end{multline}
This estimate leads to
\begin{equation}
a(P,P) \geq C \Vert P \Vert_{H^2(\Omega)^N}.
\label{4,7}
\end{equation}
Hence, due to \eqref{4,4} and \eqref{4,7}, the bilinear form $a$ defines an inner product on $H_0$, denoted by $a(\cdot, \cdot)$. This new inner product is equivalent to the traditional one of $H^2(\Omega)^N$.

Recall that $L$ is a bounded linear map.
Using the Riesz Representation Theorem for $H_0$ with the inner product $a(\cdot, \cdot)$, there exists a unique vector $W_0$ such that 
\begin{equation}
a(W_0,Q) = LQ \mbox{ for all } Q \in H_0
\end{equation}
This means 
\begin{multline}
\ds\int_{\Omega} e^{2\lambda b^{-\beta} r^{\beta}(\x)} \big( \Delta W_0 - c(\x)SW_0\big)\big( \Delta Q - c(\x)SQ \big) d\x
\\
=  - \int_{\Omega} e^{2\lambda b^{-\beta} r^{\beta}(\x)} \big( \Delta \Phi - c(\x)S\Phi 
+ {\bf F}(\x, U^{(k - 1)}, \nabla U^{(k-1)})\big)\cdot \big(
 \Delta Q - c(\x)SQ\big) d\x
\end{multline}
for all $Q \in H_0.$
It implies that
\begin{multline}
\ds\int_{\Omega} e^{2\lambda b^{-\beta} r^{\beta}(\x)} \big( \Delta W_0 - c(\x)SW_0 + \Delta \Phi - c(\x)S\Phi \\
+ {\bf F}(\x, U^{(k - 1)}, \nabla U^{(k-1)})\big)\cdot \big( \Delta Q - c(\x)SQ \big) d\x = 0
\label{4.131}
\end{multline}
for all $Q \in H_0.$

Let $U_0 = W_0 + \Phi \in H$. 
It follows from \eqref{4.131} that
\begin{equation}
\ds\int_{\Omega} e^{2\lambda b^{-\beta} r^{\beta}(\x)} \big( \Delta U_0 - c(\x)SU_0
+ {\bf F}(\x, U^{(k - 1)}, \nabla U^{(k-1)})\big)\cdot \big( \Delta Q - c(\x)SQ \big) d\x = 0
\label{4.1311}
\end{equation}
for all $Q \in H_0.$
We next claim that $U_0$ is a minimizer of the functional $J^{(k)}$ defined in \eqref{Jk}.
In fact,
for all $U \in H$, define $h = U - U_0 \in H_0$, we have
\begin{align*}
J^{(k)}&(U) - J^{(k)}(U_0) = J^{(k)}(U_0+h) - J^{(k)}(U_0)  \nonumber
\\
&= \int_{\Omega} e^{2\lambda b^{-\beta} r^{\beta}(\x)} \Big| \Delta (U_0+h)- c(\x)S(U_0+h)
+ {\bf F}\left(\x, U^{(k-1)},\nabla U^{(k-1)}\right) \Big|^2 d\x \notag
\\
&\quad \quad
 - \int_{\Omega} e^{2\lambda b^{-\beta} r^{\beta}(\x)} \Big| \Delta U_0- c(\x)SU_0
+{\bf F}\left(\x, U^{(k-1)},\nabla U^{(k-1)}\right) \Big|^2 d\x
\end{align*}
Applying $(a+b)^2 = a^2 + 2ab +b^2$, we have 
\begin{multline}
J^{(k)}(U) - J^{(k)}(U_0) 
= \int_{\Omega} e^{2\lambda b^{-\beta} r^{\beta}(\x)} |\Delta h - c(\x) S h|^2 d\x
\\
+ 2 \int_{\Omega} e^{2\lambda b^{-\beta} r^{\beta}(\x)} \Big(\Delta U_0- c(\x)SU_0
+{\bf F}\big(\x, U^{(k-1)},\nabla U^{(k-1)}\big)\Big)(\Delta h - c(\x) S h) d\x
\label{4,11}
\end{multline}
It follows from \eqref{4.1311} and \eqref{4,11} that
\[
J^{(k)}(U) - J^{(k)}(U_0) = \int_{\Omega} e^{2\lambda b^{-\beta} r^{\beta}(\x)} |\Delta h - c(\x) S h|^2 d\x \geq 0
\]
for all $U \in H$. 
As a result, $U_0$ is a minimizer of $J^{(k)}$ in $H$. 
The uniqueness of the minimizer is due to the strict convexity of $J^{(k)}$ in $H$.
\end{proof}

We next rigorously prove that the sequence of vectors, $\{U^{(k)}\}_{k\geq 1}$, 
defined in \eqref{unk} in Section \ref{iterative scheme}, converges to the true solution to \ref{main sys}. We first consider the simple case when $\Vert F \Vert_{C^1(\R,\R^d)} < \infty$. The case when $\Vert F \Vert_{C^1(\R,\R^d)} = \infty$ will follow by using a truncation technique, see Remark \ref{rem4.3}.

\begin{Theorem}  
Let $\beta_0$  be as in Lemma \ref{carleman estimate 1}. 
Fix $\beta = \beta_0$ and let $\lambda_0$ be the number as in Lemma \ref{carleman estimate 1} depending only on $\x_0,$ $\Omega,$ $d$ and $\beta$.
For all $\lambda \geq \lambda_0$,
define the sequence $\{U^{(k)}\}_{k \geq 0}$ as in \eqref{unk} in Section \ref{iterative scheme} where $U^{(0)}$ is an arbitrary function in $H^2(\Omega)^N$.
Assume that $\Vert F \Vert_{C^1(\R,\R^d)} < \infty$. Assume further that the system (\ref{main sys}) has a unique solution $U^*$. 
Then, we have
\begin{multline}
 \int_{\Omega}  e^{2\lambda b^{-\beta} r^\beta(\x)} \Big( \lambda^{-2}|D^2 ( U^{(k)} - U^*)|^2   + |\nabla (U^{(k)} - U^*)|^2 +  | U^{(k)} - U^*|^2 \Big) d\x \\
	\leq
	\Big(\frac{C}{\lambda}\Big)^k \int_{\Omega} e^{2\lambda b^{-\beta}r^{\beta}(\x)} \Big(\lambda^{-2}|D^2 ( U^{(0)} - U^*)|^2  + \big| \nabla (U^{(0)}-U^* \big) \big|^2 + \big| U^{(0)} - U^* \big|^2  \Big),
\label{main est}
\end{multline}
for $k=1, 2, \dots$ where $C$ is a constant depending only on $\x_0,$ $\Omega,$ $d$, and $\Vert F \Vert_{C^1(\R,\R^d)}$.
\label{main theorem}
\end{Theorem}

\begin{proof}
In this proof, $C$ denotes different positive constants that might change from estimate to estimate.
Define 
\begin{equation}
H_0 = \left\lbrace V \in H^2(\Omega)^N: 
\right.
V|_{\partial \Omega} = 0 \mbox{ and }
\left. \partial_{\nu}V|_{\partial \Omega} = 0 \right\rbrace.
\end{equation}
It is obvious $H_0$ a closed subspace of $H^2(\Omega)^N$.
Since $U^{(k)}$ is the minimizer of $J^{(k)}$ in $H$, by the variational principle, the following identity holds
\begin{multline}
\left\langle e^{\lambda b^{-\beta} r^{\beta}(\x)} \bigg[ \Delta U^{(k)} - c(\x)SU^{(k)}
+ {\bf F}\left(\x,U^{(k-1)},\nabla U^{(k-1)}\right)\bigg], \right.\\
\left. 
 e^{\lambda b^{-\beta} r^{\beta}(\x)} \bigg[\Delta h(\x) - c(\x)Sh \bigg]
 \right\rangle_{L^2(\Omega)^N} = 0
 \label{4.2}
\end{multline}
for all $h \in H_0$.
On the other hand, since $U^*$ is the solution to  system \eqref{main sys},
\begin{equation}
\left\langle e^{\lambda b^{-\beta} r^{\beta}(\x)} \bigg[ \Delta U^{*} - c(\x)SU^{*}
+ {\bf F}\left(\x,U^{*},\nabla U^{*}\right)\bigg], \right.
\left. 
 e^{\lambda b^{-\beta} r^{\beta}(\x)} \bigg[\Delta h - c(\x)Sh \bigg]
 \right\rangle_{L^2(\Omega)^N} = 0
 \label{4.3}
\end{equation}
for all $h \in H_0$.
Subtracting (\ref{4.3}) from (\ref{4.2}), we obtain
\begin{multline}
\Big\langle e^{\lambda b^{-\beta} r^{\beta}(\x)} \big[ \Delta (U^{(k)}-U^{*})
- c(\x)S(U^{(k)}-U^{*}) 
+ {\bf F}(\x,U^{(k-1)},\nabla U^{(k-1)}) 
\\
- {\bf F}(\x,U^{*},\nabla U^{*}) \big],
e^{\lambda b^{-\beta} r^{\beta}(\x)} \big[\Delta h - c(\x)Sh \big]
 \Big\rangle_{L^2(\Omega)^N} = 0
 \label{4.4}
\end{multline}
for all $h \in H_0$.
Using the test function $h =U^{(k)} - U^{*} \in H_0$ and  H\"older's inequality, we obtain from identity (\ref{4.4}) that
\begin{multline}
	\int_{\Omega}
e^{2\lambda b^{-\beta}r^{\beta}(\x)} |\Delta h - c(\x) S h|^2 d\x
\\
\leq
\Big[\int_{\Omega}
e^{2\lambda b^{-\beta}r^{\beta}(\x)}
|{\bf F}(\x,U^{(k-1)},\nabla U^{(k-1)}) 
- {\bf F}(\x,U^{*},\nabla U^{*})|^2d\x\Big]^{1/2}
\\
\times \Big[\int_{\Omega}
e^{2\lambda b^{-\beta}r^{\beta}(\x)} |\Delta h - c(\x) S h|^2 d\x\Big]^{1/2},
\label{4.5}
\end{multline}
which implies
\begin{equation}
	\int_{\Omega}
e^{2\lambda b^{-\beta}r^{\beta}(\x)} |\Delta h - c(\x) S h|^2 d\x
\leq
\int_{\Omega}
e^{2\lambda b^{-\beta}r^{\beta}(\x)}
|{\bf F}(\x,U^{(k-1)},\nabla U^{(k-1)}) 
- {\bf F}(\x,U^{*},\nabla U^{*})|^2d\x.
\label{4.7}
\end{equation}
Since $F$ has a finite $C^1$ norm, so does ${\bf F}$. Using the inequality $(a+b)^2 \leq 2(a^2 + b^2)$, we obtain from \eqref{4.7} that
\begin{equation}
	\int_{\Omega}
e^{2\lambda b^{-\beta}r^{\beta}(\x)} |\Delta h - c(\x) S h|^2 d\x
\leq
C\int_{\Omega}
e^{2\lambda b^{-\beta}r^{\beta}(\x)}
[|U^{(k - 1)} - U^*|^2 + |\nabla (U^{(k - 1)} - U^*)|^2]d\x.
\label{4.8}
\end{equation}
On the other hand, applying the inequality $(a-b)^2 \geq \frac{1}{2}a^2 - b^2$, we get
\begin{equation}
\int_{\Omega} e^{2\lambda b^{-\beta}r^{\beta}(\x)} |\Delta h - c(\x) S h|^2 d\x \geq \frac{1}{2}\int_{\Omega} e^{2\lambda b^{-\beta}r^{\beta}(\x)} \left| \Delta h \right|^2 d\x
- \int_{\Omega} e^{2\lambda b^{-\beta}r^{\beta}(\x)} \left| c(\x)Sh\right|^2 d\x.
\label{4.12}
\end{equation}
Using the Carleman estimate \eqref{33} for the function $h \in H_0$, we have
\begin{multline}
	\int_{\Omega} e^{2\lambda b^{-\beta} r^{\beta}(\x)}|\Delta h|^2 d\x
	\geq 
	C\lambda^{-1} \int_{\Omega}e^{2\lambda b^{-\beta} r^\beta(\x)} |D^2h|^2 d\x
	\\
	+ 	C \lambda^3 \int_{\Omega}  e^{2\lambda b^{-\beta} r^{\beta}}|h|^2 d\x
	+ C \lambda \int_{\Omega} e^{2\lambda b^{-\beta} r^{\beta}(\x)} |\nabla h |^2 d\x.
	\label{4.13}
\end{multline}
Combining \eqref{4.8}, \eqref{4.12} and \eqref{4.13}, we obtain
\begin{multline}
\lambda^{-1} \int_{\Omega}e^{2\lambda b^{-\beta} r^\beta(\x)} |D^2h|^2 d\x
	+ 	\lambda \int_{\Omega}  e^{2\lambda b^{-\beta} r^{\beta}}\Big( |h|^2 + |\nabla h|^2 \Big) d\x \\
	\leq
	C \int_{\Omega} e^{2\lambda b^{-\beta}r^{\beta}(\x)} \left( \big| U^{(k-1)} - U^* \big|^2 + \big| \nabla\left(U^{(k-1)}-U^* \right) \big|^2 \right).
\label{4.17}
\end{multline}
Multiply both sides of \eqref{4.17} with $\ds\frac{1}{\lambda}$ and recall $h = U^{(k)} - U^*$. We have
\begin{multline}
 \int_{\Omega}  e^{2\lambda b^{-\beta} r^\beta(\x)} \Big( \lambda^{-2}|D^2 ( U^{(k)} - U^*)|^2   + |\nabla (U^{(k)} - U^*)|^2 +  | U^{(k)} - U^*|^2 \Big) d\x \\
	\leq
	\frac{C}{\lambda} \int_{\Omega} e^{2\lambda b^{-\beta}r^{\beta}(\x)} \Big(\big| \nabla (U^{(k-1)}-U^* \big) \big|^2 + \big| U^{(k-1)} - U^* \big|^2  \Big),
\label{4.18}
\end{multline}
which implies
\begin{multline}
 \int_{\Omega}  e^{2\lambda b^{-\beta} r^\beta(\x)} \Big( \lambda^{-2}|D^2 ( U^{(k)} - U^*)|^2   + |\nabla (U^{(k)} - U^*)|^2 +  | U^{(k)} - U^*|^2 \Big) d\x \\
	\leq
	\frac{C}{\lambda} \int_{\Omega} e^{2\lambda b^{-\beta}r^{\beta}(\x)} \Big(\lambda^{-2}|D^2 ( U^{(k-1)} - U^*)|^2  + \big| \nabla (U^{(k-1)}-U^* \big) \big|^2 + \big| U^{(k-1)} - U^* \big|^2  \Big)
\label{4.19}
\end{multline}
By induction, from \eqref{4.19}, we have \eqref{main est}.
\end{proof}
\begin{Corollary}
Fix $\lambda$ large enough such that $\theta = C/\lambda \in (0, 1)$ where $C$ is the constant in \eqref{main est}. Then, 
\begin{multline}
\min_{\x \in \overline \Omega}\big\{e^{2\lambda b^{-\beta} r^\beta(\x)}\big\} \int_{\Omega}   \Big( \lambda^{-2}|D^2 ( U^{(k)} - U^*)|^2   + |\nabla (U^{(k)} - U^*)|^2 +  | U^{(k)} - U^*|^2 \Big) d\x \\
	\leq
	\theta^k  \max_{\x \in \overline \Omega}\big\{e^{2\lambda b^{-\beta} r^\beta(\x)}\big\} \int_{\Omega}  \Big(\lambda^{-2}|D^2 ( U^{(0)} - U^*)|^2  + \big| \nabla (U^{(0)}-U^* \big) \big|^2 + \big| U^{(0)} - U^* \big|^2  \Big)
\label{main est1}
\end{multline}
Inequality \eqref{main est1} rigorously guarantees that the sequence $\{U^{(k)}\}_{k\geq 1}$ converges to $U^*$ in $H^2(\Omega)^N$. As a result, the sequence $\big\{p^{(k)}(\x)\big\}_{k\geq 1}$ obtained in Step \ref{find source} of  Algorithm \ref{alg} converges to the true source function $p^*(\x)$ in $H^{2}(\Omega)^N$ where
\[
	p^*(\x) = \sum_{n = 1}^N u_n^*(\x)\Psi_n(0) 
	\quad \mbox{for all } \x \in \Omega.
\]
\end{Corollary}

\begin{Remark}
	One of our contributions in this paper is that we rigorously prove the convergence of the iterative scheme to the true solution of the nonlinear PDEs in $H^2(\Omega)^N$. It is an important improvement in comparison with the main theorem in \cite{NNT:arxiv2022, LocNguyen:arXiv2022} which prove the convergence in $H^1$ only.
\end{Remark}

\begin{Remark}
With the term $\theta \in (0,1)$, estimate \eqref{main est1} is similar to the one in the contraction mapping principle. This explains the title of the paper.
\end{Remark}

\begin{Remark}
In  Theorem \ref{main theorem}, we assume that the function $F$ has a finite $C^1$ norm. However,
in reality, $\|F\|_{C^1}$ might be infinity.
We can extend Theorem \ref{main theorem} for the case when $\Vert F \Vert_{C^1} = \infty$. 
Recall from the statement of Problem \ref{ISP} that  $\|u^*\|_{C^1(\overline{\Omega_T})} < M$ for some known large number $M$. Combining with (\ref{approx un}) and noting that $\|\Psi_n\|_{L^2(0, T)} = 1$, we have $\| u_m^*(\x) \|_{C^1(\overline \Omega)} \leq M\sqrt{T}$.
Define the smooth function $\chi \in C^{\infty}(\R\times\R^d)$ as follows
\begin{equation}
\chi(s,\mathbf{p}) = \left\{
\begin{array}{ll}
1  &\mbox{ if } |s| + |\mathbf{p}| < M\sqrt{T}, \\
\in (0,1)  &\mbox{ if } M\sqrt{T} \leq |s| + |\mathbf{p}| \leq 2M\sqrt{T},\\
0  &\mbox{ if } |s| + |\mathbf{p}| > 2M\sqrt{T}.
\end{array}
\right.
\end{equation}
We then set $\widetilde{\bf F} = \chi {\bf F}$. Since $|u_m^*| + |\nabla u_m^*|  < M$, the vector $U^*$ solves the following problem
\begin{equation}
	\left\{
		\begin{array}{ll}
			\Delta U(\x) - c(\x) SU(\x) + \widetilde{\bf F}(\x, U(\x), \nabla U(\x)) = 0& \x \in \Omega,
			\\
			U(\x) = G(\x) &\x \in \partial\Omega,\\
		\partial_{\nu} U(\x) = Q(\x) &\x \in \partial\Omega.
	\end{array}
	\right.
	\label{main sys modified}
\end{equation}

Thus, we can apply Algorithm \ref{alg} for (\ref{main sys modified}) to compute $U^*(\x)$ and the source function $p(\x)$.
\label{rem4.3}
\end{Remark}

\section{Numerical study}
\label{sec numerical study} 
In this section, we perform some numerical results obtained by Algorithm \ref{alg}. For simplicity purpose, we test our method in the 2-D case, i.e. d=2. 

\subsection{The forward problem}
In order to generated computationally simulated data (\ref{data}), we need to solve numerically the forward problem. Let $\widetilde{R} > R >0$ be two positive numbers. We define the domains 
\[ \widetilde{\Omega} = (-\widetilde{R},\widetilde{R})^2 \quad \mbox{ and } \quad \Omega = (-R,R)^2. \]
We first solving the following problem defined in the bigger domain $\widetilde{\Omega} \times (0,T)$
\begin{equation}
	\left\{
		\begin{array}{rcll}
			c(\x)u_t(\x, t) &=& \Delta u(\x, t) + F\left(\x,t,u(\x,t),\nabla u(\x,t)\right) &\x \in \widetilde{\Omega}, t \in (0, T)\\
			u(\x,0) &=& p(\x) & \x \in \widetilde{\Omega}, \\
			u(\x,t) &=& 0   & \x \in \partial\widetilde{\Omega}, t \in [\Finv0,T].
	\end{array}
	\right.
	\label{fw}
\end{equation}
Let $\x=(x,y) \in \widetilde{\Omega}$. In our numerical tests, the known coefficient function $c(x,y)$ is set as
\begin{equation}
c(x,y) = 1 + \frac{1}{50} \left( 3(1-x)^2 e^{-x^2-(y+1)^2}\right.
\left.
-10 \left( \frac{x}{5}-x^3-y^5 \right) e^{-x^2-y^2}-\frac{1}{3}e^{-(x+1)^2-y^2} \right).
\end{equation}
This function is a scale of the "peaks" function in Matlab. The values of coefficient function $c(x,y)$ vary in the range of $[0.8693,1.1618]$.
We solve (\ref{fw}) by the finite difference method using the explicit scheme. 
Afterward, we extract the data on the boundary of the domain $\Omega$ for the simulated data: 
\[g(\x,t) = u(\x,t) \mbox{ and } q(\x,t) = \partial_\nu u(\x,t) \mbox{ for } \x \in \partial\Omega, t \in [0,T]\]

\subsection{Implementation}
\label{sec implementation}
On the domain $\overline{\Omega}$, we arrange an $N_\x \times N_\x$ uniform grid
\[ \mathcal{G} = \Big\lbrace (x_i,y_j): x_i = -R + (i-1)h, y_j = -R + (j-1)h, 1 \leq i,j \leq N_\x \Big\rbrace \] 
where $N_\x$ is the number of grid points, $h = \ds\frac{2R}{N_\x-1}$ is the mesh spacing.
\begin{Remark}
In our computations, we set $R_1 = 6$, $R = 1$, $T = 1.5, N_\x = 80$.
\end{Remark} 

In the following, we present the implementation of Algorithm \ref{alg} to solve Problem \ref{ISP} .

{\bf Step \ref{step1}.} We employ the orthonormal basis $\{\Psi_n\}_{n\geq1}$ as in (\ref{sec basis}). The cut-off number $N$ in our approximation of $u(\x,t)$ in (\ref{approx u}) is chosen as follows. We first choose a test (Test 1 in Section \ref{sec numerical results}) as a reference test. Then, for each $N \geq 1$, we compute the error function
\begin{equation}
e_N(\x) =\Big| u^*(\x,0) - \sum_{n=1}^N u_n^*(\x) \Psi_n(0) \Big| \quad \x \in \Omega
\label{error N}
\end{equation}
and choose $N$ such that $\|e_N\|_{L^{\infty}(\Omega)}$ is small enough. 
Figure \ref{fig choose N} presents the values of the error $e_N$ when $N = 15, 35$ and $40$. It is obvious that increasing the value of $N$ reduces the error. With $N = 40$, the error is small enough. We, therefore, choose $N = 40$ for all numerical tests.
\begin{figure}[!ht]
	\begin{center}
		\subfloat[$N = 15$]{\includegraphics[width=0.3\textwidth]{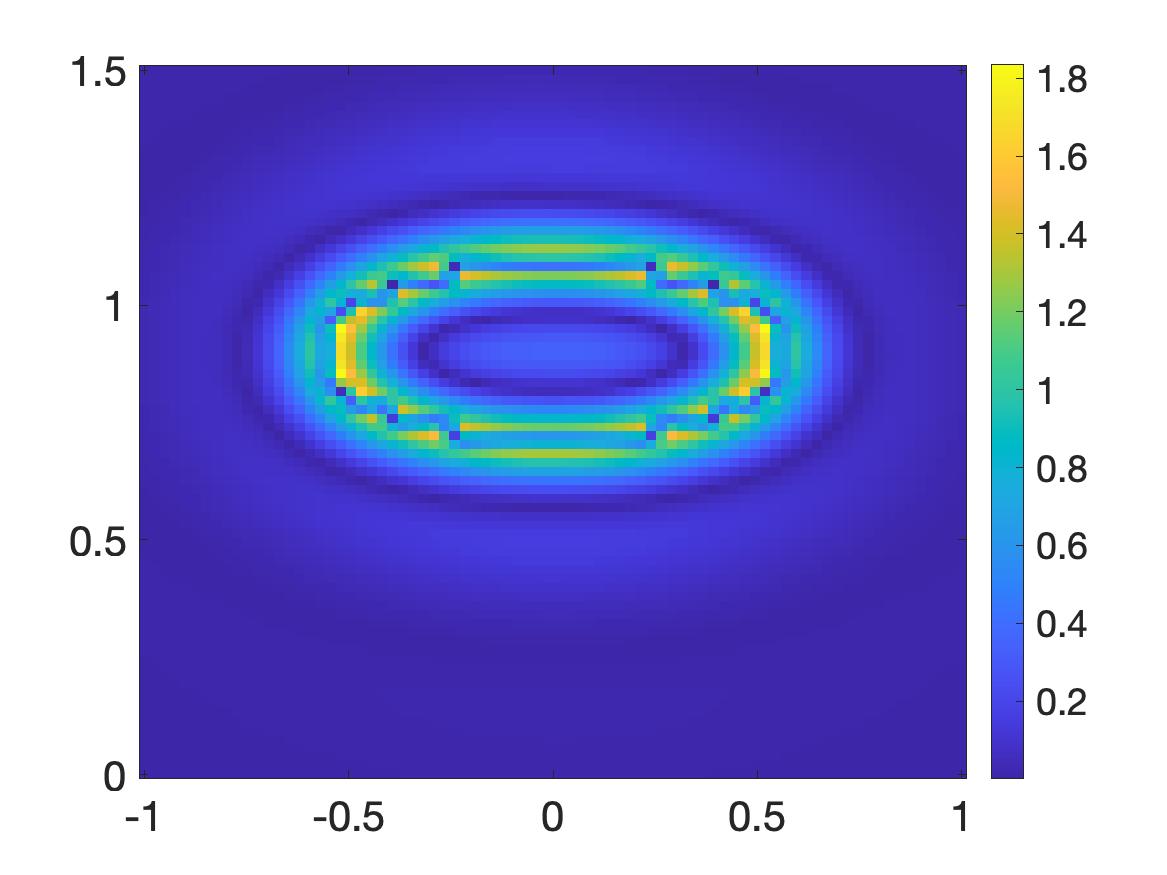}} \hfill
		\subfloat[$N = 35$]{\includegraphics[width=0.3\textwidth]{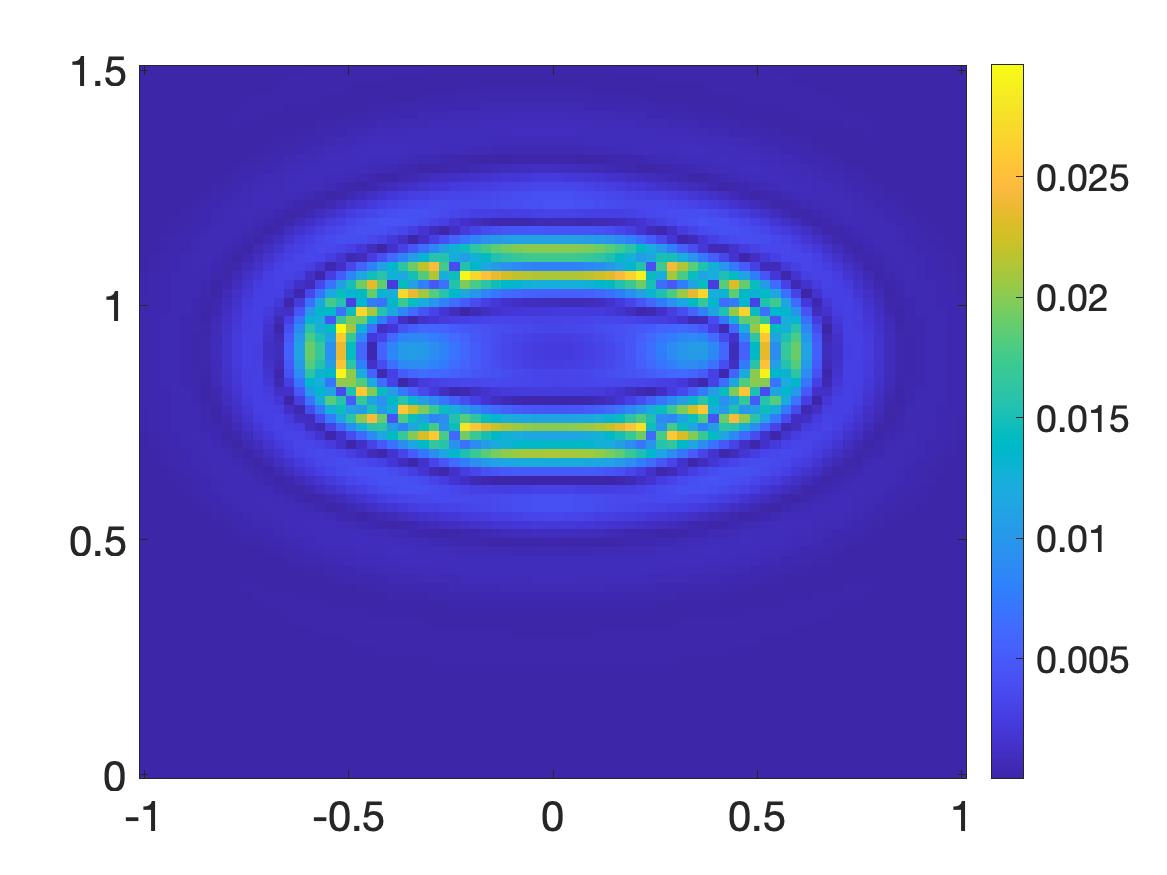}} \hfill
		\subfloat[$N = 40$]{\includegraphics[width=0.3\textwidth]{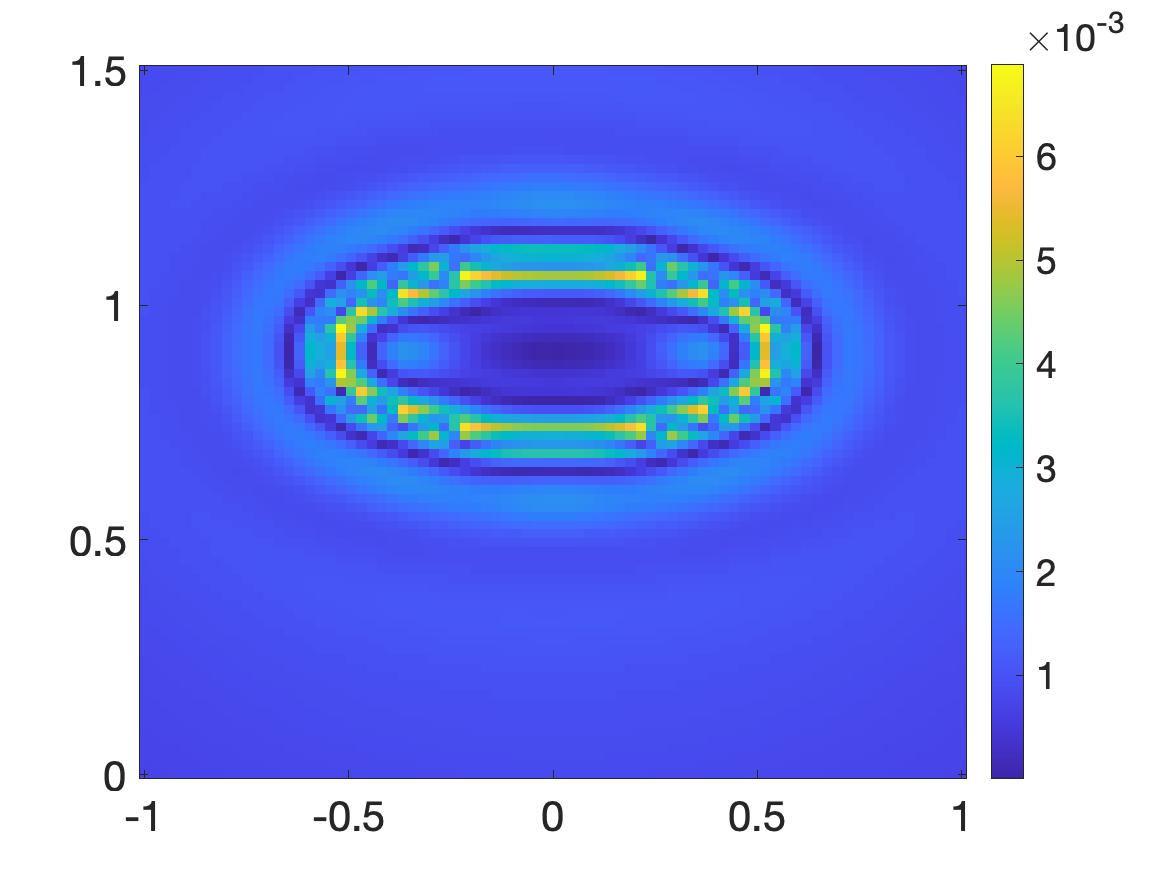}}
		\caption{\label{fig choose N} 
		The function $e_N$ defined in (\ref{error N}) where $u^*(\x,0) = p(\x)$ is the source given in Test 1. It is evident that the larger  $N$, the better approximation in (\ref{approx u}).  
}
	\end{center}
\end{figure}
\begin{Remark}
The basis $\{\Psi_n\}_{n\geq1}$ in \eqref{sec basis} was successfully used frequently in our research group. The reason for us to employ this basis rather than the well-known traditional ``sine and cosine" basis of the  Fourier series is that the first elements of the latter basis is a constant. Thus, $u_1(\x)\Psi_1'(t)$ vanishes. Hence, we might lose some information of $u_1(\x)$ in when plugging \eqref{approx ut} into \eqref{main eqn}.
This leads to a significant error in computation because the first term in the series $\sum_{n = 1}^{\infty} u_n(\x)\Psi_n(t)$ is the most important one. 
%
\end{Remark}
\textbf{Step \ref{step2}.} We compute the ``indirect data" $G(\x)$ and $Q(\x)$ on $\partial\Omega$ as in (\ref{3.8}). Denote $G^*(\x)$ and $Q^*(\x)$ as the noiseless data. The noisy data is set as follows, for $\x \in \partial\Omega$
\[ G^{\delta}(\x) = G^*(\x)\left( 1 + \delta {\rm rand} \right), \quad 
Q^{\delta}(\x) = Q^*(\x)\left( 1 + \delta{\rm rand} \right) 
\]
for all $\x \in \partial \Omega,$
where $\delta > 0$ represents for the noise level and ``rand'' is the function taking uniformly distributed random numbers in the range $[-1,1]$. In our numerical tests, a significant noise with noise level $\delta = 20\%$ is involved in the data.

\textbf{Step \ref{step3}.} For simplicity, we choose the initial solution $U^{(0)} = \overrightarrow{0}$.

\textbf{Step \ref{step4}.} Given $U^{(k-1)}$ for $k\geq 1$,
we minimize the functional $J^{(k)}$, defined in \eqref{Jk}, by the optimization package of MATLAB with the command ``lsqlin". The implementation for the quasi-reversibility method to minimize similar functionals was described in \cite[\S 5.3]{LeNguyen:jiip2020} and in \cite[\S 5]{LeNguyenNguyenPowell:jsc2021}. We do not repeat this process here.
We then set $U^{(k)} = {\rm argmin}_{V \in H} J^{(k)}(V).$

The scripts for other steps of Algorithm \ref{alg} can be written easily.

\begin{Remark}
We choose $\lambda, \beta$ in the Carleman Weight Function defined in \ref{carleman estimate} by a trial-error process that is similar to the one in \cite{LeNguyen:jiip2020}. Just as in \cite{LeNguyen:jiip2020}, we choose a reference numerical test in which we know the true solution (Test 1 in section \ref{sec numerical results}). We then tried several values of $\lambda, \beta$ until obtained a satisfactory numerical result for that reference test. Next, we have used the same values of these parameters for all other tests. In our computation, $\lambda = 40, \beta=20$.

\end{Remark}

\subsection{Numerical results}
\label{sec numerical results}
In this section, we present three (3) numerical tests to verify the efficiency of Algorithm \ref{alg}. The noise level of the data in these tests is 20\%.

\noindent{ \bf{Test 1.}} In this test, the source function is defined as follows
\[
	p_{\rm true} = \left\{
		\begin{array}{ll}
			8 & 0.2x^2 + (y - 0.2)^2 < 0.25^2,\\
			0 & \mbox{otherwise.}
		\end{array}
	\right.
\]

The function F has the form of 
\[ F(s,\mathbf{p}) =  s + \sqrt{p}\]
The true and the reconstructed source function, $p^*$ and $p_{comp}$ respectively, are displayed in Figure \protect\ref{example 1}. The true source includes an elliptic placed at point $(x,y) = (0,0.2)$ with contrast 8, see \protect\ref{m1 p true}. It is evident that our numerical method can successfully reconstruct the source function with a fast convergence. In more details,   Figure \protect\ref{m1 p comp} shows the computed source function which clearly indicates the position of the ellipse. The maximal value of the inclusion is 8.5022 (relative error 6.28\%). Although the contrast is high, our method provides a good approximation of the true source without requesting a good initial solution. Besides,  Figure \protect\ref{m1 error} illustrates the consecutive relative errors which are computed as follows
\[ err_k = \ds\frac{\Vert p^{(k)} - p^{(k - 1)}\Vert_{L^{\infty}(\Omega)}}{\Vert p^{(k)}\Vert_{L^{\infty}(\Omega)}} \quad \mbox{ for } k =1,2,\dots,8 \] 
It is noticeable that our numerical method converges rapidly. The consecutive errors go to zero quickly. Actually, after only five (5) iterations, we can obtain a stable solution for the source function. This fact verifies numerically the rate of convergence estimate $O(\theta^n)$ as $n\to \infty$ in \eqref{main est} and \eqref{main est1}.
 
\begin{figure}[!ht]
	\begin{center}
		\subfloat[\label{m1 p true}]{\includegraphics[width = 0.3\textwidth]{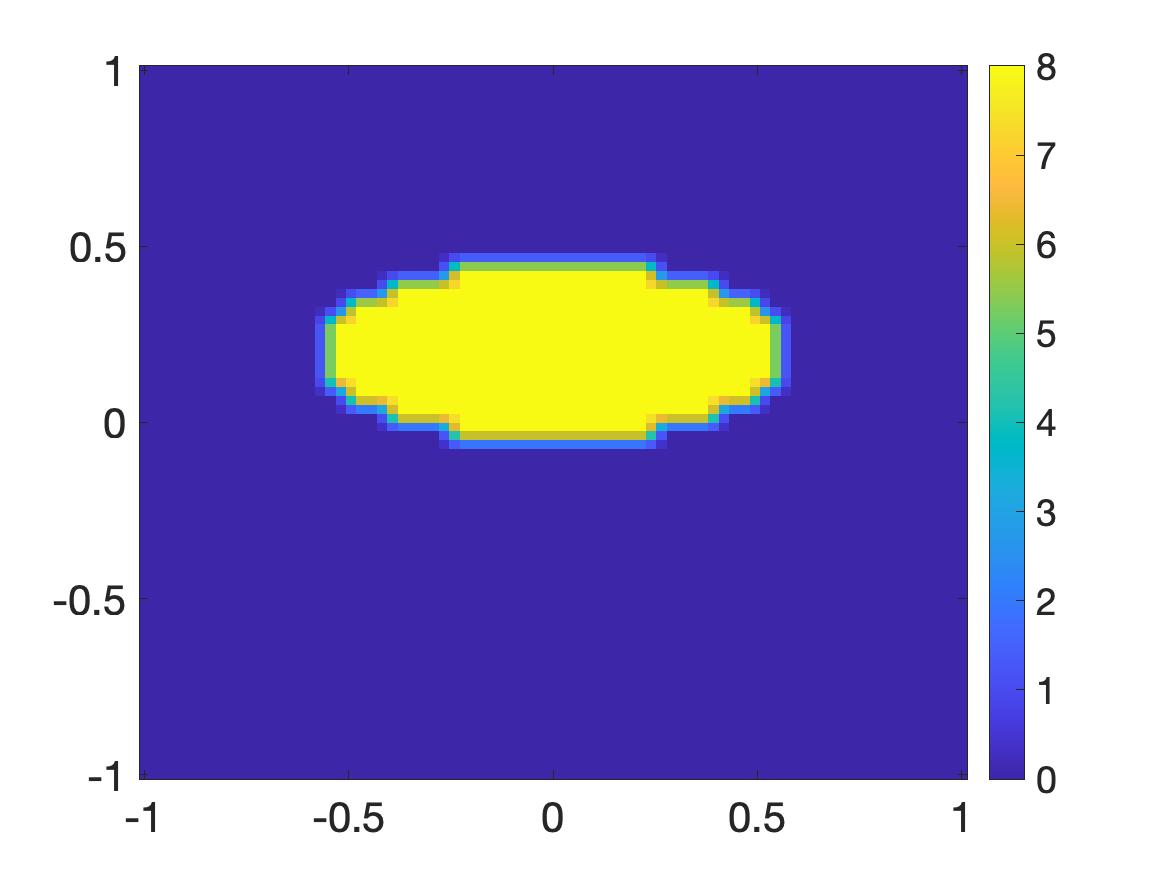}} \quad 
		\quad
		\subfloat[\label{m1 p comp}]{\includegraphics[width = 0.3\textwidth]{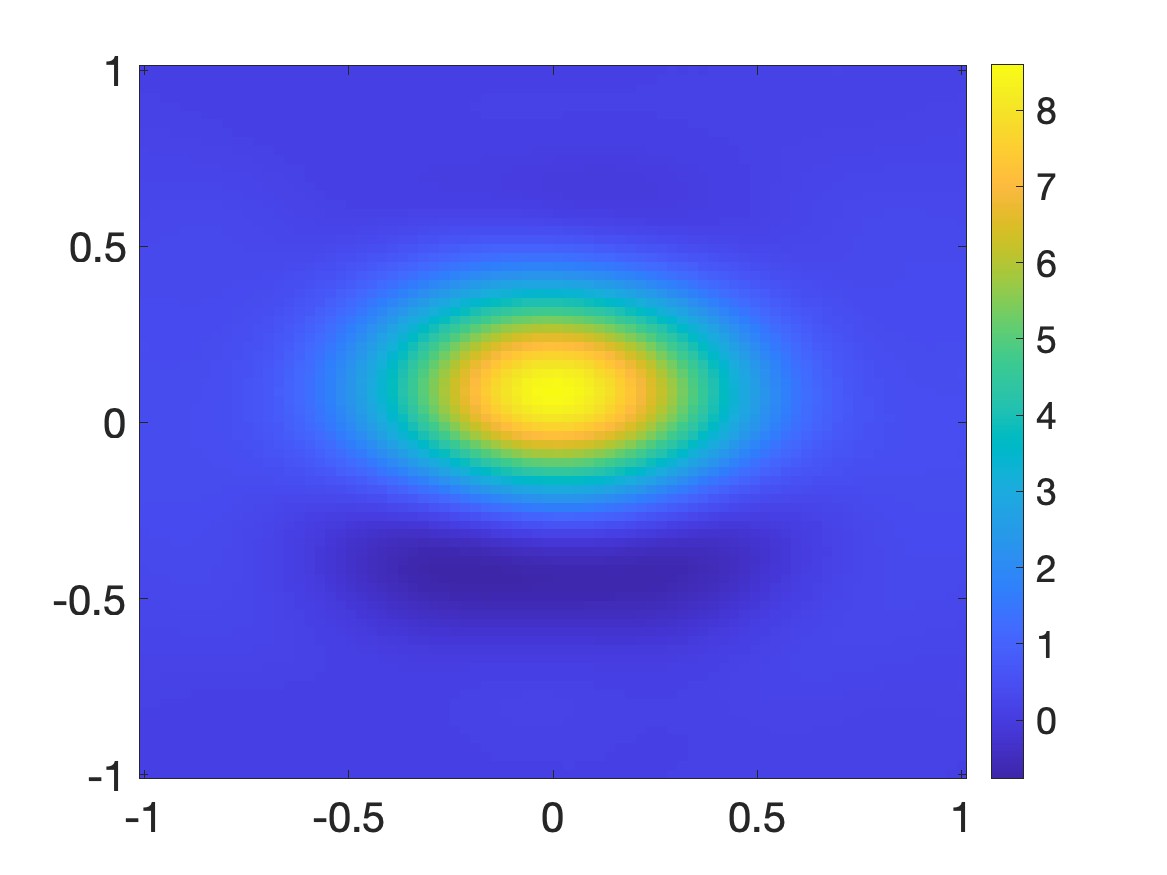}}				
		\quad
		\subfloat[\label{m1 error}]{\includegraphics[width = 0.3\textwidth]{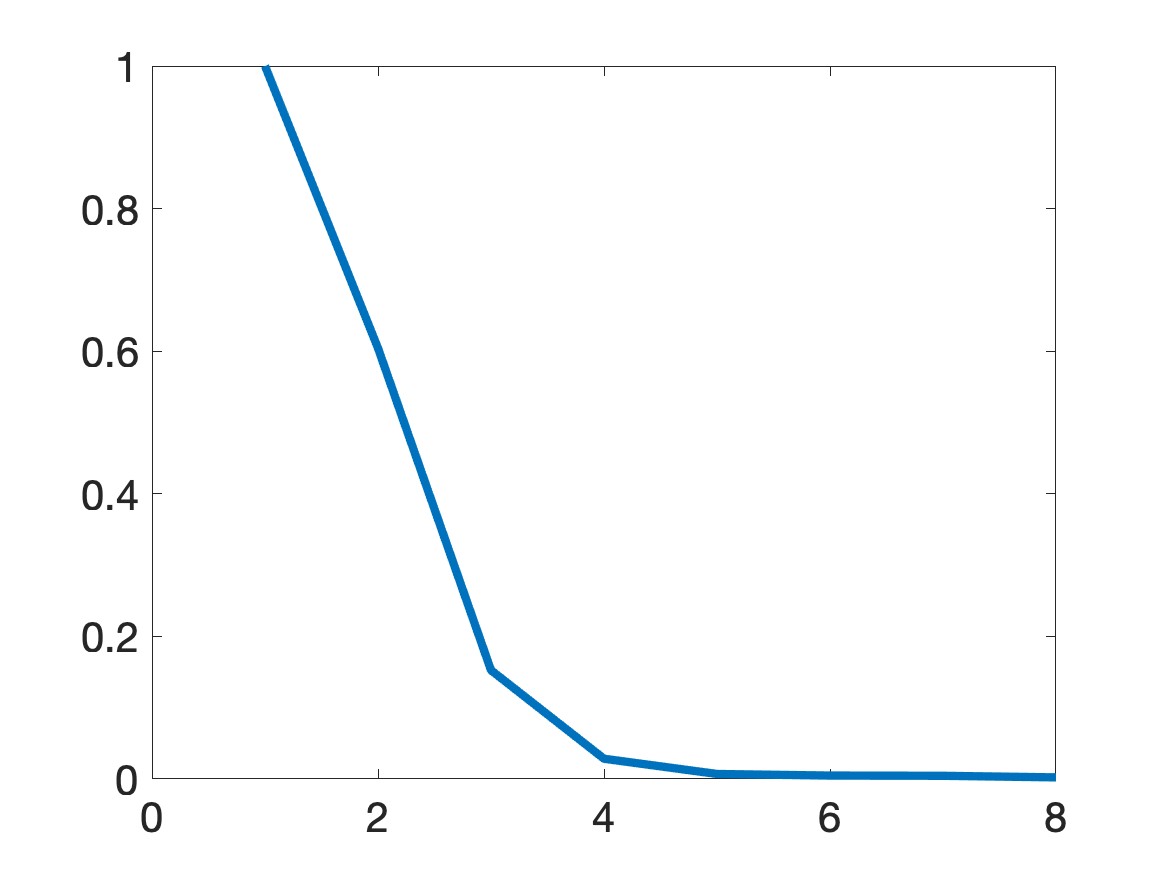}}
	\caption{\label{example 1} Test 1. The reconstruction of the source function. The initial solution of the source function $p^{(0)}(\x) = 0, \forall \x\in \Omega$. (a) The function $p_{\rm true}$.
	(b) The computed source function $p_{comp}$ obtained by Step \protect\ref{pcomp} of Algorithm \protect\ref{alg}.
	(c) The consecutive relative error $\ds\frac{\Vert p^{(k)} - p^{(k - 1)}\Vert_{L^{\infty}(\Omega)}}{\Vert p^{(k)}\Vert_{L^{\infty}(\Omega)}}$ for $k =1,2,\dots,8$.
	The noise level of the data in this test is $20\%$.
	}
	\end{center}
\end{figure}

\noindent{ \bf{Test 2.}} In this test, the source function is defined as follows
\[
	p_{\rm true} = \left\{
		\begin{array}{ll}
			6 & (x+0.5)^2 + (y + 0.5)^2 < 0.35^2,\\
			8 & (x+0.5)^2 + (y - 0.5)^2 < 0.35^2,\\
			10 & (x-0.5)^2 + (y + 0.5)^2 < 0.35^2,\\
			0 & \mbox{otherwise.}
		\end{array}
	\right.
\]

The function F has the form of 
\[ F(s,\mathbf{p}) =  s(1-s) + \frac{1}{2}\left(\big|p_1\big| - \big|p_2\big|\right)\]

The true and the reconstructed source function, $p^*$ and $p_{comp}$ respectively, are displayed in Figure \protect\ref{example 2}. The true source includes three ball with different contrasts, see \protect\ref{m2 p true}. The computed source function in Figure \protect\ref{m2 p comp} clearly indicates the position of balls and it is a good approximation of the true one. The upper left ball has the true value 8 and the maximal computed one 8.3344 (relative error 4.18\%). The lower left ball has the true value 6 and the maximal computed one 6.2892 (relative error 4.82\%). The lower right ball has the true value 10 and the maximal computed one 11.0032 (relative error 10.03\%). Besides, the Figure \protect\ref{m1 error} illustrates our numerical method converges rapidly after only five (5) iterations.

\begin{figure}[!ht]
	\begin{center}
		\subfloat[\label{m2 p true}]{\includegraphics[width = 0.3\textwidth]{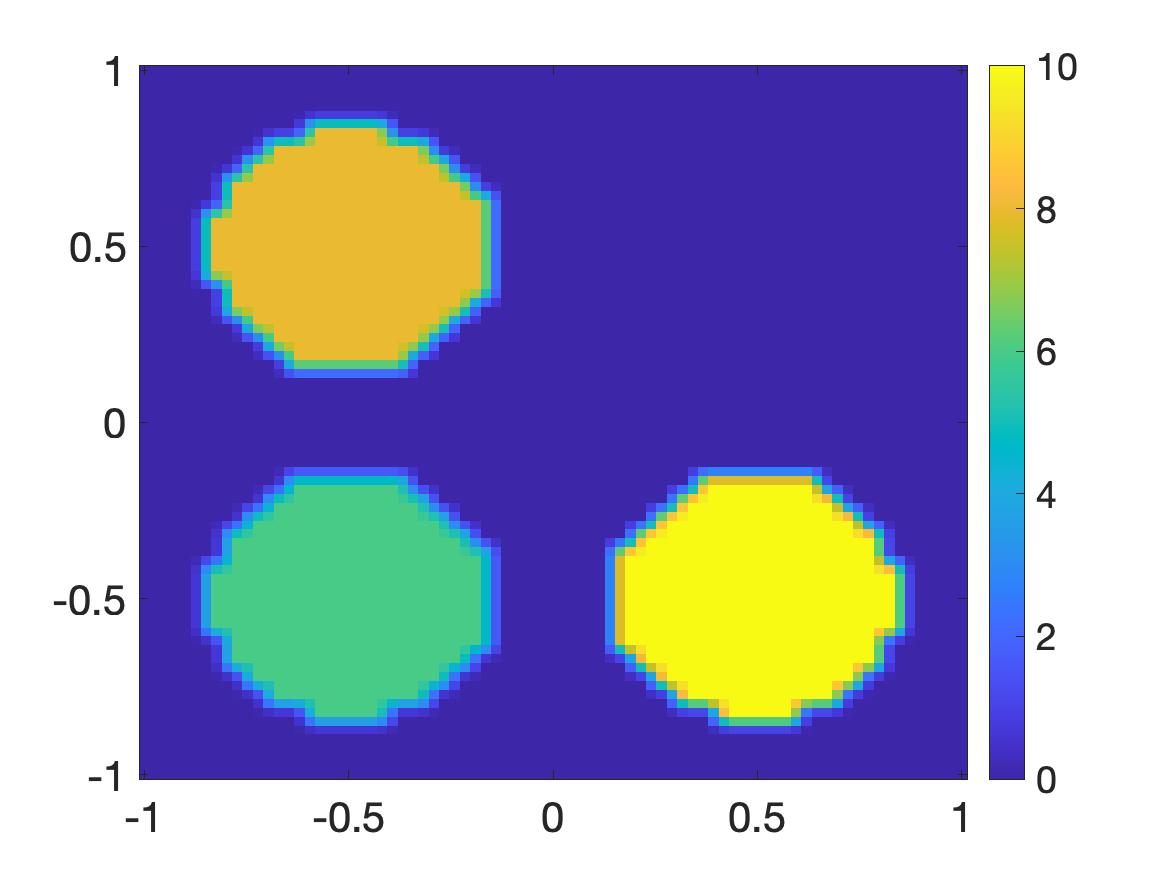}} \quad 
		\quad
		\subfloat[\label{m2 p comp}]{\includegraphics[width = 0.3\textwidth]{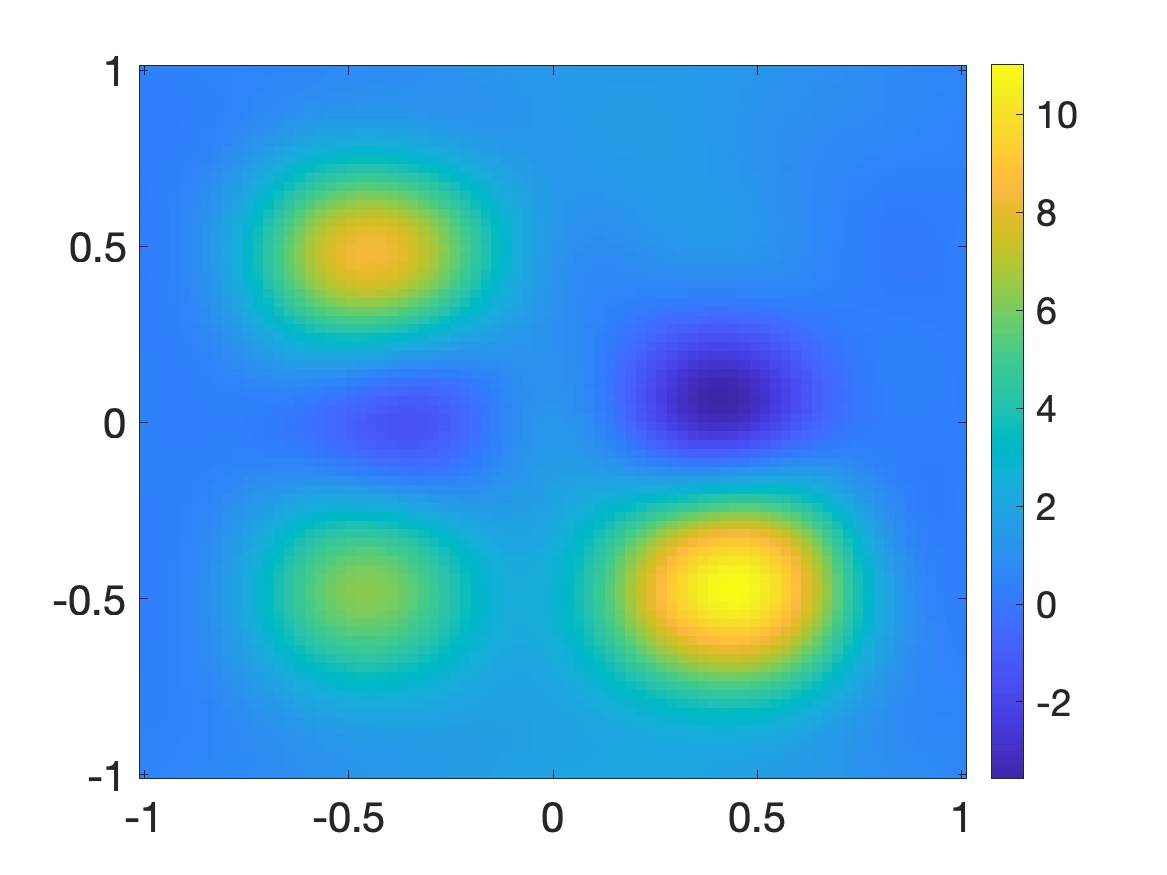}}			
		\quad
		\subfloat[\label{m2 error}]{\includegraphics[width = 0.3\textwidth]{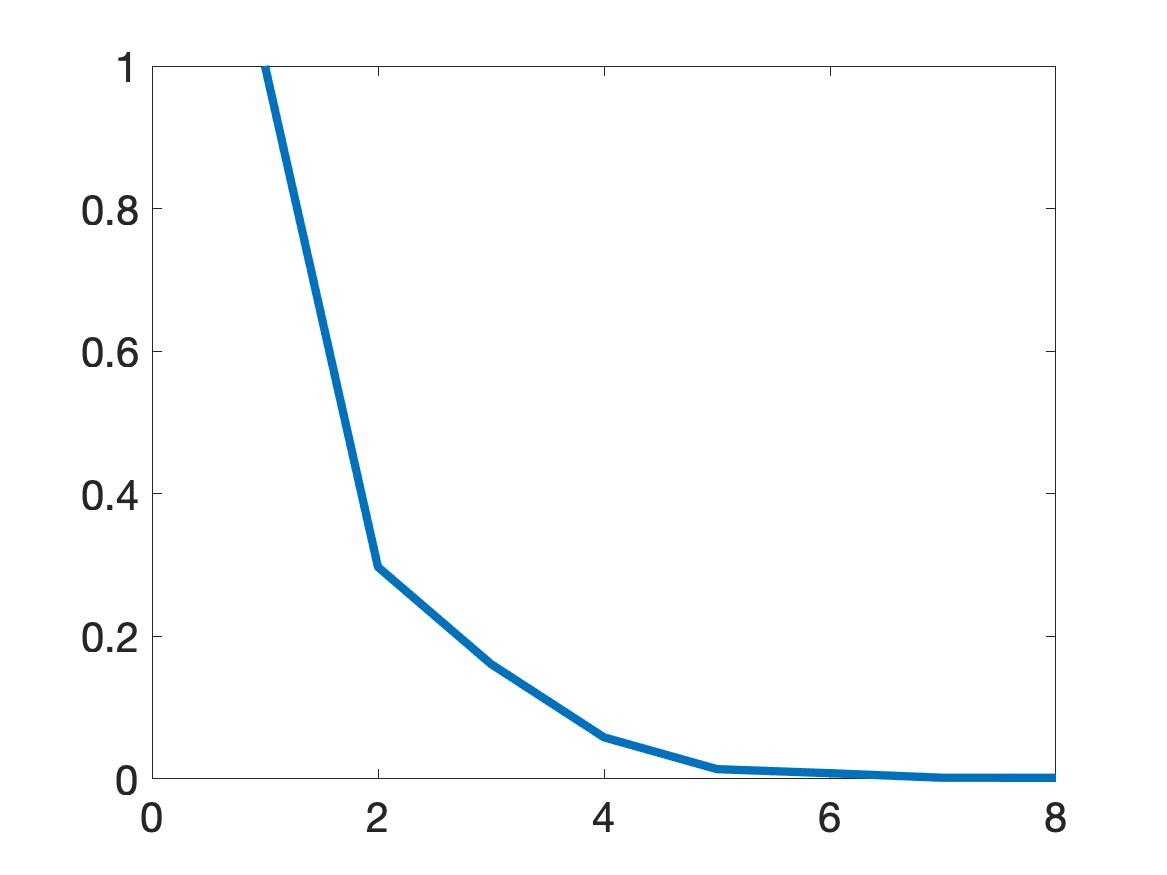}}
	\caption{\label{example 2} Test 2. The reconstruction of the source function. The initial solution of the source function $p^{(0)}(\x) = 0, \forall \x\in \Omega$. (a) The function $p_{\rm true}$.
	(b) The computed source function $p_{comp}$ obtained by Step \protect\ref{pcomp} of Algorithm \protect\ref{alg}.
	(c) The consecutive relative error $\ds\frac{\Vert p^{(k)} - p^{(k - 1)}\Vert_{L^{\infty}(\Omega)}}{\Vert p^{(k)}\Vert_{L^{\infty}(\Omega)}}$ for $k =1,2,\dots,8$.
	The noise level of the data in this test is $20\%$.
	}
	\end{center}
\end{figure} 

\noindent{ \bf{Test 3.}}
In this test, the source function is defined as follows
\[
	p_{\rm true} = \left\{
		\begin{array}{ll}
			1 & 0.4^2 < x^2 + y^2 < 0.8^2,\\
			0 & \mbox{otherwise.}
		\end{array}
	\right.
\]

The function F has the form of 
\[ F(s,\mathbf{p}) =   \sqrt{p_1^2 + p_2^2 + 1}\]
\begin{figure}[!ht]
	\begin{center}
		\subfloat[\label{m3 p true}]{\includegraphics[width = 0.3\textwidth]{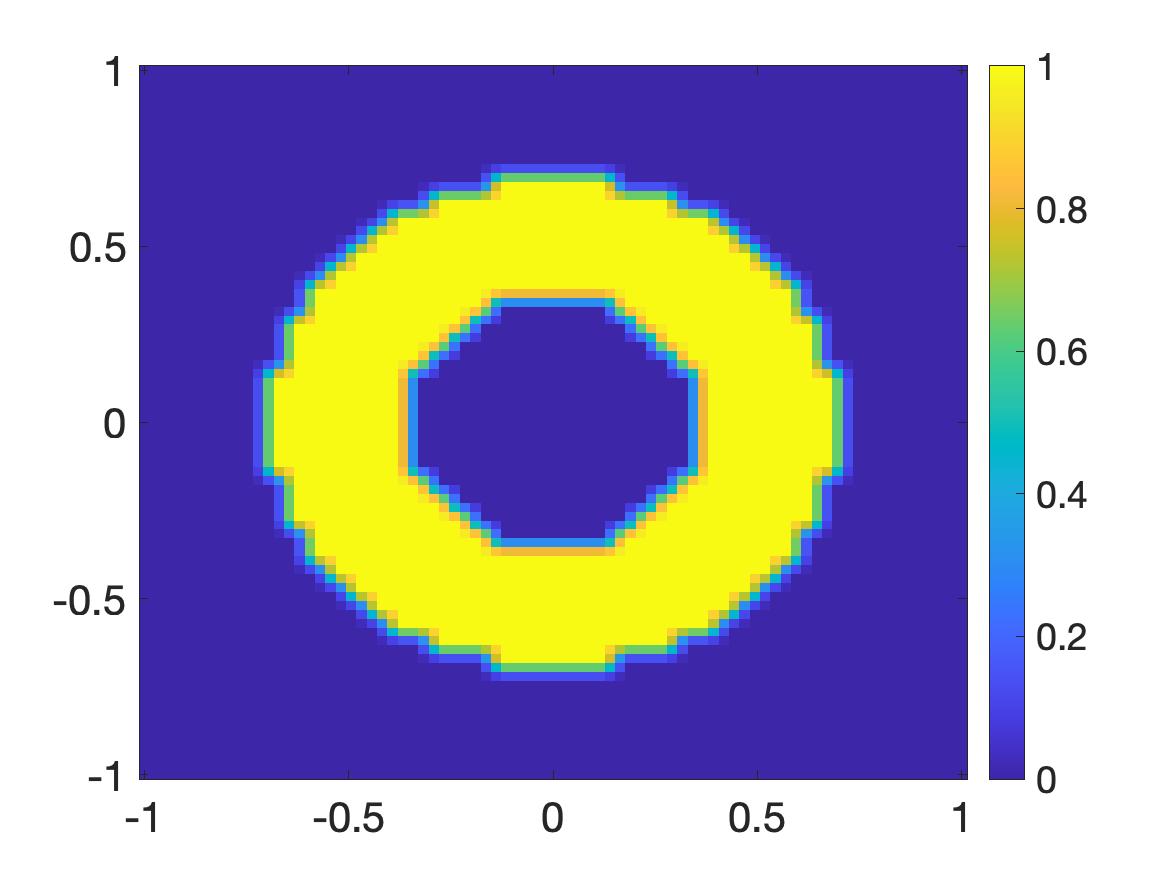}} \quad 
		\quad
		\subfloat[\label{m3 p comp}]{\includegraphics[width = 0.3\textwidth]{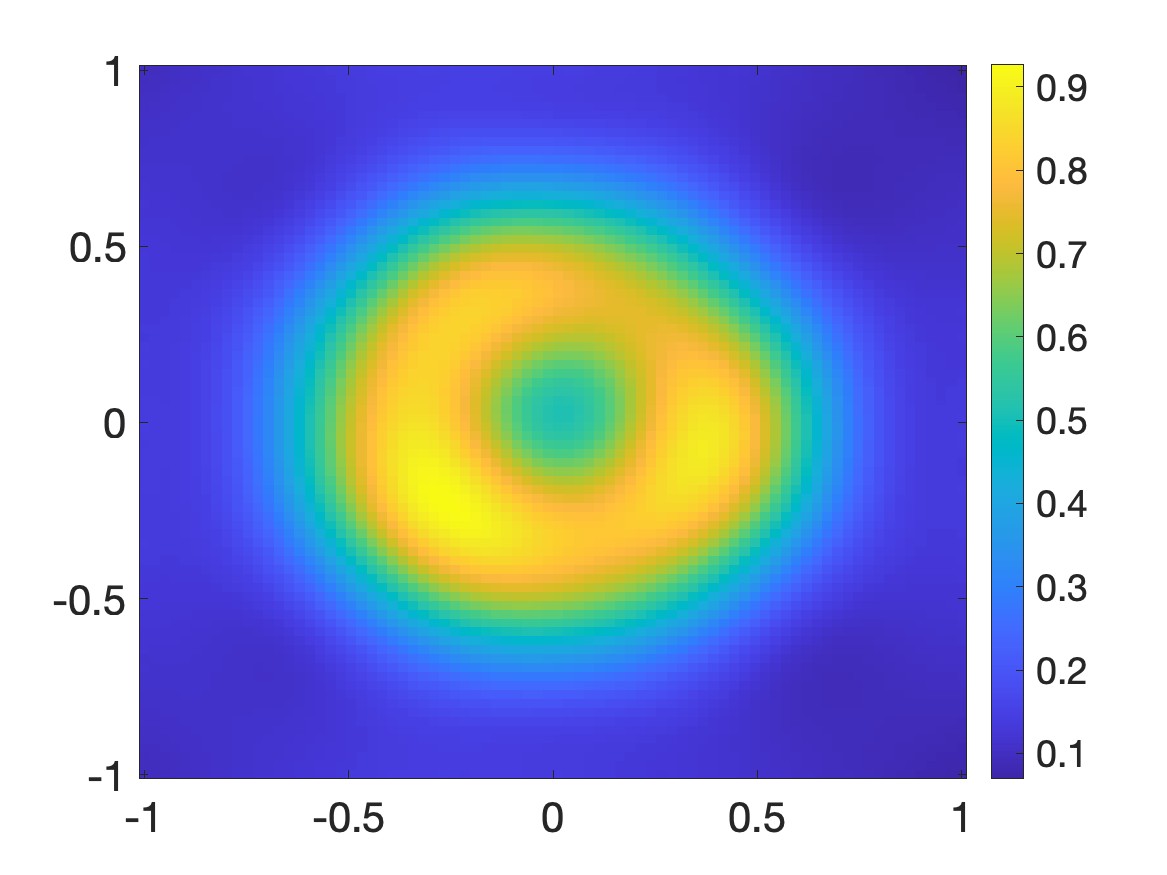}}		
		\quad
		\subfloat[\label{m3 error}]{\includegraphics[width = 0.3\textwidth]{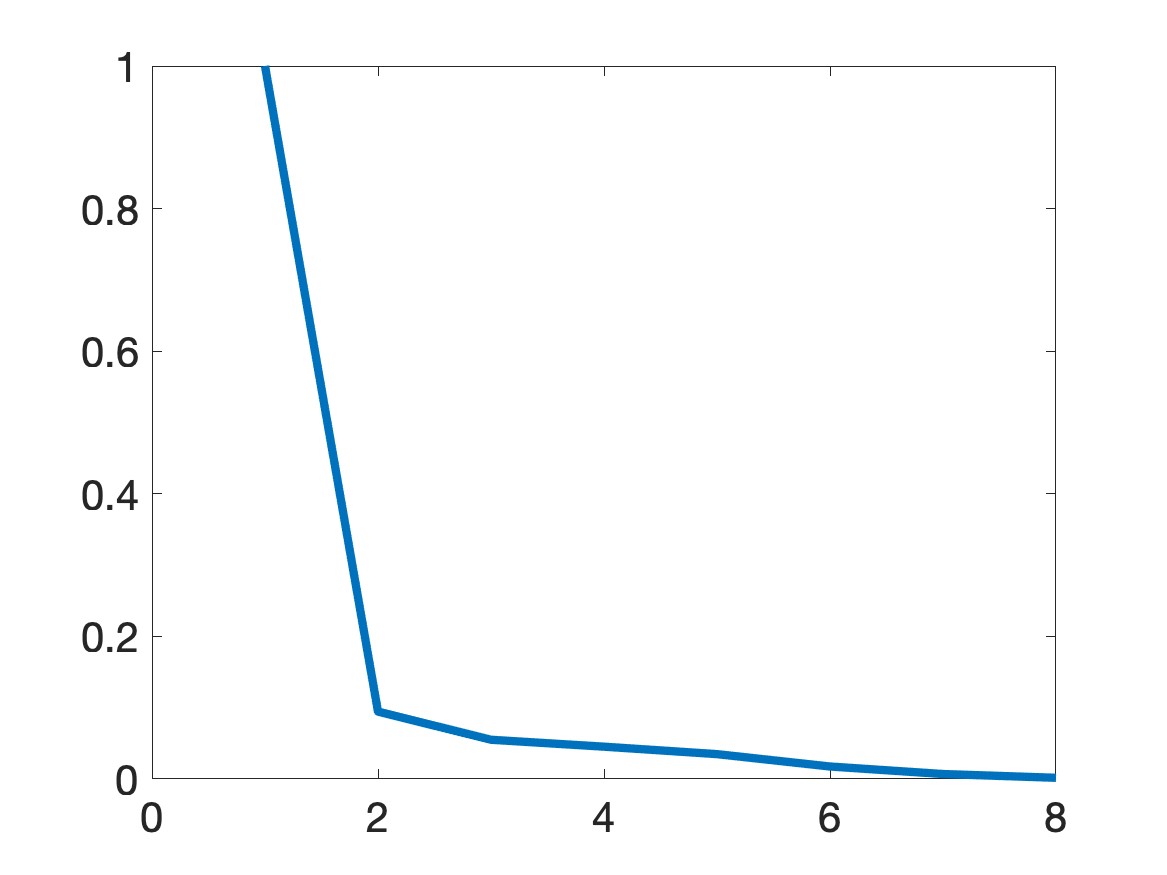}}
	\caption{\label{example 3} Test 3. The reconstruction of the source function. The initial solution of the source function $p^{(0)}(\x) = 0, \forall \x\in \Omega$. (a) The function $p_{\rm true}$.
	(b) The computed source function $p_{comp}$ obtained by Step \protect\ref{pcomp} of Algorithm \protect\ref{alg}.
	(c) The consecutive relative error $\ds\frac{\Vert p^{(k)} - p^{(k - 1)}\Vert_{L^{\infty}(\Omega)}}{\Vert p^{(k)}\Vert_{L^{\infty}(\Omega)}}$ for $k =1,2,\dots,8$.
	The noise level of the data in this test is $20\%$.
	}
	\end{center}
\end{figure} 

The true and the reconstructed source function, $p^*$ and $p_{comp}$ respectively, are displayed in Figure \protect\ref{example 3}. The true source includes a ring with contrast 1, see \protect\ref{m3 p true}. The  Figure \protect\ref{m3 p comp} shows the computed source function which clearly indicates the position of the ring and the void inside. The maximal value of the ring is 0.9226 (relative error 7.74\%). Besides, the consecutive relative errors in the Figure \protect\ref{m1 error} illustrates the fast convergence of our numerical method. The convergence is obtained after seven (7) iterations.

%
%
%

\section{Concluding remarks} 
\label{sec remarks}
In this paper, we introduce a new iterative method to solve an inverse source problem for a nonlinear parabolic equations.
The first step of our numerical method is to derive a system of nonlinear elliptic equations whose solutions yields directly the solution for the inverse source problem. We then propose an iterative scheme to solve that nonlinear system. Our numerical method is global convergent in the sense that: 
\begin{enumerate}
\item[(i)] It provides a good approximation to the true source function. 
\item[(ii)] It does not require any knowledge of the true source function. This means a good initial guess is not necessary.
\end{enumerate}
In our computation, we start our iterative scheme from the zero vector. The convergence of our scheme was proved rigorously. Numerical results are present to verify our theoretical part.

\vspace{0.2in}

\noindent{\bf Acknowledgment.}
The author sincerely appreciates Dr. Loc H. Nguyen for many fruitful discussions that strongly improve the mathematical results and the presentation of this paper.

This work  was partially supported by US Army Research Laboratory and US Army
Research Office grant W911NF-19-1-0044, by National Science Foundation grant DMS-2208159, and
by funds provided by the Faculty Research Grant program at UNC Charlotte Fund No. 111272.

\bibliographystyle{plain}
\bibliography{research}

\end{document}